\documentclass{article}
\usepackage{amsmath,amsthm,amssymb,xypic,array}
\usepackage[active]{srcltx}

\theoremstyle{plain}  

\newtheorem{theorem}{Theorem}[section]

\newtheorem{proposition}[theorem]{Proposition}
\newtheorem{lemma}[theorem]{Lemma}   
\newtheorem{corollary}[theorem]{Corollary}

\usepackage{stmaryrd}
\usepackage{amsfonts}
\usepackage{amsmath}
\usepackage{amssymb}
\usepackage[bookmarksnumbered,plainpages,hypertex]{hyperref}
\usepackage{graphicx}
\usepackage{fancyhdr}
\providecommand{\U}[1]{\protect\rule{.1in}{.1in}}
\providecommand{\U}[1]{\protect\rule{.1in}{.1in}}
\providecommand{\U}[1]{\protect\rule{.1in}{.1in}}
\providecommand{\U}[1]{\protect\rule{.1in}{.1in}}
\providecommand{\U}[1]{\protect\rule{.1in}{.1in}}
\input{xy}
\xyoption{all}
\newtheorem{remark}[theorem]{Remark}

\theoremstyle{definition}

\theoremstyle{remark}

\DeclareMathOperator{\Ker}{Ker}

\newcommand{\QED}{\ifhmode\unskip\nobreak\fi\quad {\rm Q.E.D.}} 

\newcommand{\R}{\mathbb{R}}

\newcommand{\be}{\begin{equation}}	
\newcommand{\ee}{\end{equation}}	
\newcommand{\bern}{\begin{eqnarray}}	
\newcommand{\eern}{\end{eqnarray}}	

\newcommand{\norm}[1]{\left\|#1\right\|}		

\begin{document}

\title{\sc Concentration of solutions for a \\singularly perturbed Neumann problem \\in non smooth domains}

\date{}

\maketitle

\begin{center}
\author{Serena Dipierro}
\end{center}

\begin{center}
{\small SISSA, Sector of Mathematical Analysis\\
Via Bonomea 265, 34136 Trieste, Italy\\
E-mail address: dipierro@sissa.it
}
\end{center}

\begin{abstract}
We consider the equation  $-\epsilon^ {2}\Delta u + u = u^ {p}$ 
in a bounded domain $\Omega\subset\R^{3}$ with edges. 
We impose Neumann boundary conditions, 
assuming $1<p<5$, and prove 
concentration of solutions at suitable points of $\partial\Omega$ on the edges.
\end{abstract}

\section{Introduction}

In this paper we study the following singular perturbation problem with Neumann boundary condition 
in a bounded domain $\Omega\subset\R^{3}$ whose boundary $\partial\Omega$ is non smooth:
\bern
\left\{ 
\begin{array}{ll} 
         -\epsilon^ {2}\Delta u + u = u^ {p}   \quad & \mathrm{in\ } \Omega,      \\
          \frac{\partial u}{\partial \nu} = 0 \qquad  & \mathrm{on\ } \partial\Omega .
 \end{array} 
\right.         
 \label{problem}\eern 
Here $p\in\left(1, 5\right)$ is subcritical and $\nu$ denotes the outer unit normal at $\partial\Omega$.

\medskip

Problem $(\ref{problem})$ or some of its variants 
arise in several physical and biological models. 
Consider, for example, the Nonlinear Schr\"{o}dinger Equation
\be     i\hbar\frac{\partial\psi}{\partial t}   =    - \frac{\hbar^{2}}{2m} \Delta\psi + V\psi -  \gamma\vert\psi\vert^{p-2}\psi,   \label{nse}\ee
where $\hbar$ is the Planck constant, $V$ is the potential, and $\gamma$ and $m$ 
are positive constants. 
Then standing waves of $(\ref{nse})$ can be found 
setting $\psi\left(x, t\right)=e^{-iEt/\hbar}v\left(x\right)$, 
where $E$ is a constant and the real function $v$ satisfies the elliptic equation 
\be     -\hbar^{2}\Delta v  +  \tilde{V} v  =  \vert v\vert^{p-2}v   \nonumber\ee
for some modified potential $\tilde{V}$. 
In particular, when one considers the semiclassical 
limit $\hbar\rightarrow 0$, 
the last equation becomes a singularly perturbed one; 
see for example \cite{AM}, \cite{FW}, 
and references therein.

Concerning reaction-diffusion systems, this phenomenon is 
related to the so-called Turing's instability. 
More precisely, it is known that scalar reaction-diffusion equations 
in a convex domain admit only constant stable steady state solutions; 
see \cite{CH}, \cite{Mat}. 
On the other hand, as noticed in \cite{Tu}, 
reaction-diffusion systems with different 
diffusivities might generate 
non-homogeneous stable steady states. 
A well-known example is the Gierer-Meinhardt system, 
introduced in \cite{GM} to describe some biological experiment.  
We refer to \cite{Ni}, \cite{NTY} for more details.

\bigskip

The study of the concentration phenomena at points 
for smooth domains is very rich and has been intensively developed 
in recent years. 
The search for such condensing solutions 
is essentially carried out by two methods. 
The first approach is variational and uses tools 
of the critical point theory or topological methods. 
A second way is to reduce the problem to a finite-dimensional 
one by means of Lyapunov-Schmidt reduction.

The typical concentration behavior of solution $U_{Q, \epsilon}$ 
to $(\ref{problem})$ is via a scaling of the variables in the form 
\be    U_{Q,\epsilon}\left(x\right)\sim U\left(\frac{x-Q}{\epsilon}\right), \label{appr}\ee
where $Q$ is some point of $\bar{\Omega}$, 
and $U$ is a solution of the problem 
\be       -\Delta U + U = U^ {p}   \quad  \mathrm{in\ } \R^{3}    
               \quad  \mathrm{(or\ in\ } \R^{3}_{+}=\left\lbrace \left(x_{1}, x_{2}, x_{3}\right)\in\R^{3}  :  x_{3}>0\right\rbrace),
\label{prob}\ee   
the domain depending on whether $Q$ lies in the interior of $\Omega$ 
or at the boundary; 
in the latter case Neumann conditions are imposed. 
When $p<5$ (and indeed only if this inequality is satisfied), 
problem $(\ref{prob})$ admits positive radial solutions 
which decay to zero at infinity; see \cite{BL}, \cite{St}. 
Solutions of $(\ref{problem})$ with this profile 
are called \textit{spike-layers}, 
since they are highly concentrated near some point of $\bar{\Omega}$. 

Let us recall some known results. 
\textit{Boundary-spike layers} are solutions of $(\ref{problem})$ 
with a concentration at one or more points of 
the boundary $\partial\Omega$ as $\epsilon\rightarrow 0$. 
They are peaked near critical point of the mean curvature. 
It was shown in \cite{NT1}, \cite{NT2} that 
mountain-pass solutions of $(\ref{problem})$ 
concentrate at $\partial\Omega$ near 
global maxima of the mean curvature.  
One can see this fact considering the variational structure of $(\ref{problem})$. 
In fact, its solutions can be found as critical points 
of the following Euler-Lagrange functional 
\be
 \tilde{I}_{\epsilon}\left(u\right)  = \frac{1}{2} \int_{\Omega}\left(\epsilon^ {2}\vert\nabla u\vert^ {2} + u^ {2}\right)dx - \frac{1}{p+1}\int_{\Omega}\vert u\vert^ {p+1}dx, \quad  u\in W^{1,2}\left(\Omega\right).
\nonumber\ee
Plugging into $\tilde{I}_{\epsilon}$ a function of the form 
$(\ref{appr})$ with $Q\in\partial\Omega$ 
one sees that 
\be    \tilde{I}_{\epsilon}\left(U_{Q, \epsilon}\right)=C_{0}\epsilon^{3}-C_{1}\epsilon^{4}H\left(Q\right)+o\left(\epsilon^{4}\right), \label{mean}\ee 
where $C_{0}, C_{1}$ are positive constants depending only 
on the dimension and $p$, and $H$ is the mean curvature; 
see for instance \cite{AM}, Lemma $9.7$. 
To obtain this expansion one can use the radial symmetry of $U$ 
and parametrize $\partial\Omega$ as a normal graph near $Q$. 
From the above formula one can see that the bigger 
is the mean curvature the lower is the energy of this function: 
roughly speaking, boundary spike layers would tend to move 
along the gradient of $H$ in order to minimize their energy. 
Moreover one can say that the energy of spike-layers is of order $\epsilon^{3}$, 
which is proportional to the volume of their \textit{support}, 
heuristically identified with a ball of radius 
$\epsilon$ centered at the peak. 
There is an extensive literature regarding the search of 
more general solutions of $(\ref{problem})$ 
concentrating at critical points of $H$; 
see \cite{DFW}, \cite{Gr}, \cite{GPW}, \cite{Gu}, \cite{Li}, \cite{LNT}, \cite{NPT}, \cite{We}. 
 
There are other types of solutions of $(\ref{problem})$ 
with interior and/or boundary peaks, 
possible multiple, 
which are constructed by using gluing techniques or 
topological methods; 
see \cite{DW}, \cite{DY}, \cite{GW}, \cite{GW1}, \cite{GWW}, \cite{Wa}. 
For \textit{interior spike} solutions 
the distance function $d$ from the boundary $\partial\Omega$ 
plays a role similar to that of the mean curvature $H$. 
In fact, solutions with interior peaks, 
as for the problem with the Dirichlet boundary condition, 
concentrate at critical points of $d$, 
in a generalized sense; 
see \cite{LN}, \cite{NW}, \cite{We1}. 

Concerning a singularly perturbed problem with mixed Dirichlet 
and Neumann boundary conditions, in \cite{GMMP1}, \cite{GMMP2} 
it was proved that, under suitable geometric conditions on 
the boundary of a smooth domain, 
there exist solutions which approach the intersection 
of the Neumann and the Dirichlet parts as the singular perturbation 
parameter tends to zero.

There is an extensive literature regarding this type of problems, 
but only the case $\Omega$ smooth was considered. 
Concerning the case $\Omega$ non smooth, 
at our knowledge there is only a bifurcation result for the equation 
\bern
\left\{ 
\begin{array}{ll} 
         \Delta u + \lambda f\left(u\right)  = 0   \quad & \mathrm{in\ } \Omega,     \\
          \frac{\partial u}{\partial \nu} = 0 \qquad  & \mathrm{on\ } \partial\Omega,
 \end{array} 
\right.         
 \nonumber\eern
obtained by Shi in \cite{Sh} 
when $\Omega$ is a rectangle $\left(0, a\right)\times\left(0, b\right)$ in $\R^{2}$.

\medskip

In this paper we consider the problem $(\ref{problem})$, 
where $\Omega$ is a bounded domain in $\R^{3}$ 
whose boundary $\partial\Omega$ has smooth edges. 
If we denote by $\Gamma$ an edge of $\partial\Omega$, 
we can consider the function $\alpha:\Gamma\rightarrow\R$ 
which associates to every $Q\in\Gamma$ the opening angle at $Q$, $\alpha\left(Q\right)$. 
As in the previous case, we can expect that the function $\alpha$ 
plays the same role as the mean curvature $H$ for a smooth domain. 
In fact, plugging into $\tilde{I}_{\epsilon}$ a function of the form 
$(\ref{appr})$ with $Q\in\Gamma$ 
one obtains an expression similar to $(\ref{mean})$, 
with $C_{0}\alpha\left(Q\right)$ instead of $C_{0}$; 
see Lemma \ref{lem:espansione}. 
Roughly speaking, we can say that the energy of 
solutions is of order $\epsilon^{3}$, 
which is proportional to the volume of their \textit{support}, 
heuristically identified with a ball of radius $\epsilon$ 
centered at the peak $Q\in\Gamma$; 
then, when we intersect this ball 
with the domain we obtain the dependence on the angle $\alpha\left(Q\right)$.

The main result of this paper is the following
\begin{theorem} 
Let $\Omega\subset\R^{3}$  be a piecewise smooth bounded domain 
whose boundary $\partial\Omega$ has a finite number of smooth edges, and $1<p<5$. 
Fix an edge $\Gamma$, and suppose $Q\in\Gamma$ 
is a local strict maximum or minimum of the function $\alpha$, with $\alpha\left(Q\right)\neq\pi$. 
Then for $\epsilon >0$ sufficiently small problem 
$(\ref{problem})$ admits a solution concentrating at $Q$.
\label{th:solution}
\end{theorem} 

\begin{remark} 
The condition that $Q$ is a local strict maximum or minimum of $\alpha$ 
can be replaced by the fact that there exists an open set $V$ of $\Gamma$ 
containing $Q$ such that $\alpha\left(Q\right)>\sup_{\partial V}\alpha$ 
or $\alpha\left(Q\right)<\inf_{\partial V}\alpha$.  
\end{remark}

\begin{remark} 
The condition $\alpha\left(Q\right)\neq\pi$ is natural 
since it is needed to ensure that $\partial\Omega$ 
is not flat at $Q$.
\end{remark}

\begin{remark} 
We expect a similar result to hold in higher dimension, 
with substantially the same proof. 
For simplicity we only treat the $3$-dimensional case.   
\end{remark}

\noindent
The general strategy for proving Theorem $\ref{th:solution}$ 
relies on a finite-dimensional reduction; 
see for example the book \cite{AM}.

By the change of variables $x\mapsto\epsilon x$, problem $(\ref{problem})$ can be transformed into
\bern
\left\{ 
\begin{array}{ll} 
         -\Delta u + u = u^ {p}   \quad & \mathrm{in\ } \Omega_{\epsilon},     \\
          \frac{\partial u}{\partial \nu} = 0 \qquad  & \mathrm{on\ } \partial\Omega_{\epsilon},
 \end{array} 
\right.         
 \label{problem1}\eern 
where $\Omega_{\epsilon}:=\frac{1}{\epsilon}\Omega$. 
Solutions of  $(\ref{problem1})$ can be found as critical points  
of the Euler-Lagrange functional
\be
 I_{\epsilon}\left(u\right)  = \frac{1}{2} \int_{\Omega_{\epsilon}}\left(\vert\nabla u\vert^ {2} + u^ {2}\right)dx - \frac{1}{p+1}\int_{\Omega_{\epsilon}}\vert u\vert^ {p+1}dx, \quad  u\in W^{1,2}\left(\Omega_{\epsilon}\right).
  \label{var1}
\ee
Now, first of all, one finds a manifold $Z_{\epsilon}$ 
of approximate solutions to the given problem, 
which are of the form $U_{Q,\epsilon}\left(x\right)=\varphi_{\mu}\left(\epsilon x\right)U\left(x-Q\right)$, 
where $\varphi_{\mu}$ is a suitable cut-off function 
defined in a neighborhood of $Q\in\Gamma$; 
see the beginning of Section $4$, Lemma $\ref{lem:pseudo}$.

To apply the  
method described in Subsection $2.1$
one needs the condition that the 
critical manifold $Z_{\epsilon}$ is non-degenerate, 
in the sense that it satisfies property $ii)$ in Subsection $2.1$. 
The result of non-degeneracy in $\Omega_{\epsilon}$, 
obtained in Lemma $\ref{lem:nondeg}$, 
follows from the non-degeneracy of 
a manifold $Z$ of critical points of the unperturbed problem 
in $K=\tilde{K}\times\R\subset\R^{3}$, 
where $\tilde{K}\subset\R^{2}$ is a cone of opening angle $\alpha\left(Q\right)$. 
In fact, one sees that $\Omega_{\epsilon}$ tends to $K$ 
as $\epsilon\rightarrow 0$. 
To show the non-degeneracy of the unperturbed manifold $Z$ 
we follow the line of Lemma $4.1$ in the book \cite{AM} or 
Lemma $3.1$ in \cite{Ma}. 
We prove that $\lambda=0$ is a simple eigenvalue 
of the linearized of the unperturbed problem at $U\in Z$; 
see Lemma $\ref{lem:cono}$. 
Moreover, if $\alpha\left(Q\right)<\pi$, 
it has only one negative simple eigenvalue; 
whereas, if $\alpha\left(Q\right)>\pi$, 
it has two negative simple eigenvalues; 
see Corollary $\ref{cor:nondeg}$. 
We note that in the case $\alpha\left(Q\right)=\pi$, 
that is when $\partial\Omega$ is flat at $Q$, 
$\lambda=0$ is an eigenvalue of 
multiplicity $2$. 
The proof relies on Fourier analysis, 
but in this case one needs spherical functions 
defined on a portion of the sphere instead of the whole $S^{2}$. 

Then one solves 
the equation up to a vector parallel to the tangent plane 
of the manifold $Z_{\epsilon}$, 
and generates a new manifold $\tilde{Z}_{\epsilon}$ 
close to $Z_{\epsilon}$ which represents a natural constraint 
for the Euler functional $(\ref{var1})$; 
see the proof of Proposition $\ref{prop:ridotto}$. 
By \textit{natural constraint} we mean a set 
for which constrained critical points of $I_{\epsilon}$ 
are true critical points. 

We can finally apply the above mentioned perturbation method 
to reduce the problem to a finite dimensional one, 
and study the functional constrained on $\tilde{Z}_{\epsilon}$. 
Lemma $\ref{lem:espansione}$ provides an expansion of the 
energy of the approximate solution peaked at $Q$ and 
allows us to see that the dominant term in the expression 
of the reduced functional at $Q$ is $\alpha\left(Q\right)$. 
This implies Theorem $\ref{th:solution}$.

\bigskip

The paper is organized in the following way. 
In Section $2$ we collect preliminary material: 
we recall the abstract variational perturbative scheme and 
obtain some useful geometric results. 
In Section $3$ we prove the non-degeneracy 
of the critical manifold for the unperturbed 
problem in the cone $K$. 
In Section $4$ we construct the manifold of approximate solutions, 
showing that it is a non-degenerate pseudo-critical manifold, 
expand the functional on the natural constraint 
and deduce Theorem $\ref{th:solution}$.

\medskip 

\subsection*{Notation} 
Generic fixed constant will be denoted by $C$, 
and will be allowed to vary within a single line or formula. 
The symbols $o_{\epsilon}\left(1\right)$, 
$o_{R}\left(1\right)$ $o_{\epsilon, R}\left(1\right)$ will denote respectively 
a function depending on $\epsilon$ that tends to $0$ as $\epsilon\rightarrow 0$, 
a function depending on $R$ that tends to $0$ as $R\rightarrow +\infty$ 
and a function depending on both $\epsilon$ and $R$ that tends to $0$ 
as $\epsilon\rightarrow 0$ and $R\rightarrow +\infty$. 
We will work in the space $W^{1,2}\left(\Omega_{\epsilon}\right)$, 
endowed with the norm $\norm{u}^{2} = \int_{\Omega_{\epsilon}}\left(\vert\nabla u\vert^ {2} + u^ {2}\right)dx$, 
which we denote simply by $\norm{u}$, 
without any subscript.

\section{Some preliminaries}

In this section we introduce the abstract perturbation method 
which takes advantage of the variational structure of the problem, 
and allows us to reduce it to a finite dimensional one. 
We refer the reader mainly to \cite{AM}, \cite{Ma} 
and the bibliography therein. 

In the second part we make some computations 
concerning the parametrization of $\partial\Omega$ 
and $\partial\Omega_{\epsilon}$, 
and in particular of the edge.

\subsection{Perturbation in critical point theory} 
In this subsection we recall some results 
about the existence of critical points 
for a class of functionals which are perturbative in nature. 
Given an Hilbert space $H$, 
which might depend on the perturbation parameter $\epsilon$, 
let $I_{\epsilon} : H\rightarrow\R$ be a functional of class $C^{2}$ 
which satisfies the following properties
\begin{itemize}
\item[i)]  there exists a smooth finite-dimensional manifold, 
compact or not, $Z_{\epsilon}\subseteq H$ 
such that $\norm{I'_{\epsilon}(z)}\leq C\epsilon$ for every $z\in Z_{\epsilon}$ and for some fixed constant $C$, independent of $z$ and $\epsilon$; 
moreover $\norm{I''_{\epsilon}(z)\left[ q\right] }\leq C\epsilon\norm{q}$ 
for every $z\in Z_{\epsilon}$ and every $q\in T_{z}Z_{\epsilon}$;
\item[ii)] letting $P_{z} : H\rightarrow \left(T_{z}Z_{\epsilon}\right)^{\perp}$, for every $z\in Z_{\epsilon}$, 
be the projection onto the orthogonal complement of $T_{z}Z_{\epsilon}$, 
there exists $C>0$, independent of $z$ and $\epsilon$, such that $P_{z}I''_{\epsilon}(z)$, 
restricted to $\left(T_{z}Z_{\epsilon}\right)^{\perp}$, is invertible from $\left(T_{z}Z_{\epsilon}\right)^{\perp}$ 
into itself, and the inverse operator satisfies $\norm{\left( P_{z}I''_{\epsilon}(z)\right)^{-1}}\leq C$.
\end{itemize}   
\noindent We assume that $Z_{\epsilon}$ has a local $C^{2}$ 
parametric representation $z=z_{\xi}$, $\xi\in\R^{d}$. 
If we set $W=\left(T_{z}Z_{\epsilon}\right)^{\perp}$, 
we look for critical points of  $I_{\epsilon}$ in the form $u=z+w$ 
with $z\in Z_{\epsilon}$ and $w\in W$. 
If $P_{z} : H\rightarrow W$ is as in $ii)$, 
the equation $I'_{\epsilon}\left(z+w\right) = 0$ 
is equivalent to the following system
\bern
\left\{ 
\begin{array}{ll} 
          P_{z}I'_{\epsilon}\left(z+w\right) = 0  \qquad & \mathrm{\left(\textit{the\ auxiliary\ equation}\right),}      \\
           \left(Id-P_{z}\right) I'_{\epsilon}\left(z+w\right) = 0  \quad & \mathrm{\left(\textit{the\ bifurcation\ equation}\right).}
 \end{array} 
\right.         
 \label{aux_bif}\eern 
\begin{proposition} 
\textsl{(See Proposition $2.2$ in \cite{Ma})} 
Let $i),ii)$ hold. 
Then there exists $\epsilon_{0}>0$ with the following property: 
for all $\vert\epsilon\vert<\epsilon_{0}$ and for all $z\in Z_{\epsilon}$, 
the auxiliary equation in $(\ref{aux_bif})$ has a unique solution $w=w_{\epsilon}(z)$ such that:
\begin{itemize}
\item[j)] $w_{\epsilon}(z)\in W$ is of class $C^{1}$ with respect to $z\in Z_{\epsilon}$ 
and  $w_{\epsilon}(z)\rightarrow 0$ as $\vert\epsilon\vert\rightarrow 0$, 
uniformly with respect to $z\in Z_{\epsilon}$, 
together with its derivative with respect to $z$,  $w'_{\epsilon}$;
\item[jj)] more precisely one has that $\norm{w_{\epsilon}(z)}= O\left(\epsilon\right)$ as $\epsilon\rightarrow 0$, for all $z\in Z_{\epsilon}$.
\end{itemize}   
\label{pr:fi}\end{proposition}

\noindent We shall now solve the bifurcation equation in $(\ref{aux_bif})$. 
In order to do this, let us define the \textit{reduced functional} $\Phi_{\epsilon} : Z_{\epsilon}\rightarrow\R$ by setting 
$\Phi_{\epsilon}(z) = I_{\epsilon}(z+ w_{\epsilon}(z))$.
\begin{theorem}
\textsl{(See Theorem $2.3$ in \cite{Ma})} 
Suppose we are in the situation of Proposition $\ref{pr:fi}$, 
and let us assume that $\Phi_{\epsilon}$ has, for $\vert\epsilon\vert$ sufficiently small, 
a critical point $z_{\epsilon}$. 
Then $u_{\epsilon} = z_{\epsilon} + w(z_{\epsilon})$ is a critical point of $I_{\epsilon}$.
\label{th:rid}
\end{theorem}

\noindent The next result is a useful criterion for applying Theorem $\ref{th:rid}$, 
based on expanding $I_{\epsilon}$ on $Z_{\epsilon}$ in powers of $\epsilon$.

\begin{theorem} 
\textsl{(See Theorem $2.4$ in \cite{Ma})} 
Suppose the assumptions of Proposition $\ref{pr:fi}$ hold, 
and that for $\epsilon$ small there is a local parametrization $\xi\in\frac{1}{\epsilon}U\subseteq\R^{d}$ of $Z_{\epsilon}$ 
such that, as $\epsilon\rightarrow 0$, $I_{\epsilon}$ admits the expansion 
$I_{\epsilon}(z_{\xi})=C_{0}+\epsilon G(\epsilon\xi)+o(\epsilon)$, 
for $\xi\in\frac{1}{\epsilon}U$, 
for some function $G : U \rightarrow\R$. 
Then we still have the expansion 
$\Phi_{\epsilon}(z_{\xi})=C_{0}+\epsilon G(\epsilon\xi)+o(\epsilon)$, 
as $\epsilon\rightarrow 0$.
Moreover, if $\bar{\xi}\in U$ is a strict local maximum or minimum of $G$, 
then for $\vert\epsilon\vert$ small the functional $I_{\epsilon}$ has a critical point $u_{\epsilon}$. 
Furthermore, if $\bar{\xi}$ is isolated, we can take $u_{\epsilon}-z_{\bar{\xi}/\epsilon}=o(1/\epsilon)$ 
as $\epsilon\rightarrow 0$.
\label{th:espansione}
\end{theorem}

\begin{remark} 
The last statement asserts that, once we scale back in $\epsilon$, 
the solution concentrates near $\bar{\xi}$.
\end{remark}

\subsection{Geometric preliminaries}

Let us describe $\partial\Omega_{\epsilon}$ near a generic point $Q$ 
on the edge $\Gamma$ of $\partial\Omega_{\epsilon}$. 
Without loss of generality, we can assume that $Q=0\in\R^{3}$, 
that $x_{1}$-axis is the tangent line at $Q$ to $\Gamma$ 
in $\partial\Omega_{\epsilon}$, or $\partial\Omega$. 
In a neighborhood of $Q$, let $\gamma:\left( -\mu_{0}, \mu_{0}\right) \rightarrow\R^{2}$ 
be a local parametrization of $\Gamma$, 
that is $\left(x_{2}, x_{3}\right)  = \gamma\left(x_{1}\right) = \left(\gamma_{1}\left(x_{1}\right), \gamma_{2}\left(x_{1}\right)\right)$. 
Then one has, for $\vert x_{1}\vert <\mu_{0}$, 
\bern 
    \left(x_{2}, x_{3}\right)&=&\gamma\left(x_{1}\right) \nonumber\\
                                        &=&\gamma\left(0\right) +\gamma'\left(0\right)x_{1} + \frac{1}{2}\gamma''\left(0\right)x_{1}^{2} + 
                                         O\left(\vert x_{1}\vert^{3}\right) \nonumber\\
                                        &=&\frac{1}{2}\gamma''\left(0\right)x_{1}^{2} + O\left(\vert x_{1}\vert^{3}\right).
\nonumber\eern
On the other hand, $\Gamma$ is parametrized by 
$\left(x_{2}, x_{3}\right) = \gamma_{\epsilon}\left(x_{1}\right) := \frac{1}{\epsilon}\gamma\left(\epsilon x_{1}\right)$, 
for which the following expansions hold
\bern 
    \gamma_{\epsilon}\left(x_{1}\right) &=& \frac{\epsilon}{2}\gamma''\left(0\right)x_{1}^{2} + O\left(\epsilon^{2}\vert x_{1}\vert^{3}\right), \nonumber\\
    \frac{\partial\gamma_{\epsilon}}{\partial x_{1}} &=& \epsilon\gamma''\left(0\right)x_{1} + O\left(\epsilon^{2}\vert x_{1}\vert^{2}\right).
 \label{exp}\eern

Now we introduce a new set of coordinates on $B_{\frac{\mu_{0}}{\epsilon}}\left(Q\right) \bigcap\Omega_{\epsilon}$: 
\be       y_{1}=x_{1},          \qquad    \left(y_{2}, y_{3}\right) = \left(x_{2}, x_{3}\right) - \gamma_{\epsilon}\left(x_{1}\right). \nonumber\ee
The advantage of these coordinates is that the edge identifies with $y_{1}$-axis, 
but the corresponding metric $g = \left(g_{ij}\right)_{ij}$ will not be flat anymore. 
If $\gamma_{\epsilon}\left(x_{1}\right)=\left(\gamma_{\epsilon 1}\left(x_{1}\right), \gamma_{\epsilon 2}\left(x_{1}\right)\right)$, 
the coefficients of $g$ are given by
\bern \left(g_{ij}\right) = \left(\frac{\partial x}{\partial y_{i}}\cdot\frac{\partial x}{\partial y_{j}}\right) =
\left( 
\begin{array}{ccc} 
    1 +  \frac{\partial\gamma_{\epsilon 1}}{\partial y_{1}}\frac{\partial\gamma_{\epsilon 1}}{\partial y_{1}}  
    + \frac{\partial\gamma_{\epsilon 2}}{\partial y_{1}}\frac{\partial\gamma_{\epsilon 2}}{\partial y_{1}} 
    & \frac{\partial\gamma_{\epsilon 1}}{\partial y_{1}} & \frac{\partial\gamma_{\epsilon 2}}{\partial y_{1}}    \\ 
    \frac{\partial\gamma_{\epsilon 1}}{\partial y_{1}}  & 1 & 0       \\
    \frac{\partial\gamma_{\epsilon 2}}{\partial y_{1}}  & 0 & 1
\end{array} 
\right).
\nonumber\eern
From the estimates in $(\ref{exp})$ it follows that
\be       g_{ij} = Id + \epsilon A + O\left(\epsilon^{2}\vert x_{1}\vert^{2}\right),      \label{gij}   \ee
where
\bern A =
\left( 
\begin{array}{cc} 
    0 & \gamma''\left(0\right)x_{1}    \\ 
    \gamma''\left(0\right)^{T}x_{1}    & 0 
\end{array} 
\right).
\nonumber\eern
It is also easy to check that the inverse matrix $\left(g^{ij}\right)$ 
is of the form $g^{ij} = Id - \epsilon A + O\left(\epsilon^{2}\vert x_{1}\vert^{2}\right)$. 
Furthermore one has $\det g = 1$. 
Therefore, by $(\ref{gij})$, for any smooth function $u$ there holds
\bern 
     \Delta_{g}u = \Delta u -\epsilon \left[2\left(\gamma''\left(0\right)y_{1}\cdot\nabla_{\left( y_{2}, y_{3}\right) }\frac{\partial u}{\partial y_{1}}\right) + \left(\gamma''\left(0\right)\cdot\nabla_{\left( y_{2}, y_{3}\right)}u\right) \right]   \nonumber\\
  + O\left(\epsilon^{2}\vert x_{1}\vert^{2}\right) \vert\nabla^{2}u\vert +  O\left(\epsilon^{2}\vert x_{1}\vert^{2}\right)\vert\nabla u\vert.
\label{lapl} 
\eern
      
\bigskip

Now, let us consider a smooth domain $\tilde{\Omega}\subset\R^{3}$ and 
$\tilde{\Omega}_{\epsilon}=\frac{1}{\epsilon}\tilde{\Omega}$. 
In the same way we can describe $\partial\tilde{\Omega}_{\epsilon}$ 
near a generic point $Q\in\partial\tilde{\Omega}_{\epsilon}$. 
Without loss of generality, we can assume that $Q=0\in\R^{3}$, 
that $\left\lbrace x_{3}=0\right\rbrace$ is the tangent plane 
of $\partial\tilde{\Omega}_{\epsilon}$, or $\partial\tilde{\Omega}$, at $Q$, 
and that the outer normal $\nu\left(Q\right) = \left(0, 0, -1\right)$. 
In a neighborhood of $Q$, let $x_{3} = \psi\left(x_{1}, x_{2}\right)$ 
be a local parametrization of $\partial\tilde{\Omega}$. 
Then one has, for $\vert\left(x_{1}, x_{2}\right)\vert <\mu_{1}$,
\bern 
    x_{3} &=& \psi\left(x_{1}, x_{2}\right)  \nonumber\\
              &=& \frac{1}{2}\left(A_{Q}\left(x_{1}, x_{2}\right)\cdot\left(x_{1}, x_{2}\right)\right) + 
                 C_{Q}\left(x_{1}, x_{2}\right)  + O\left(\vert\left(x_{1}, x_{2}\right)\vert^{4}\right),
\nonumber\eern
where $A_{Q}$ is the Hessian of $\psi$ at $\left(0, 0\right)$ and $C_{Q}$ is 
a cubic polynomial, which is given precisely by
\be C_{Q}\left(x_{1}, x_{2}\right) =  \frac{1}{6}\sum_{i,j,k=1}^{2}\frac{\partial^{3}\psi}{\partial x_{i}\partial x_{j}\partial x_{k}}\left(0, 0\right)x_{i}x_{j}x_{k}. \nonumber \ee
On the other hand, $\partial\tilde{\Omega}_{\epsilon}$ is parametrized by 
$x_{3}= \psi_{\epsilon}\left(x_{1}, x_{2}\right) := \frac{1}{\epsilon} \psi\left(\epsilon x_{1}, \epsilon x_{2}\right)$, 
for which the following expansions hold
\bern 
    \psi_{\epsilon}\left(x_{1}, x_{2}\right) &=& \frac{\epsilon}{2}\left(A_{Q}\left(x_{1}, x_{2}\right)\cdot\left(x_{1}, x_{2}\right)\right)   + 
    \epsilon^{2}C_{Q}\left(x_{1}, x_{2}\right) +  O\left(\epsilon^{3}\vert\left(x_{1}, x_{2}\right)\vert^{4}\right), \nonumber\\
     \frac{\partial\psi_{\epsilon}}{\partial x_{i}}\left(x_{1}, x_{2}\right) &=& \epsilon\left(A_{Q}\left(x_{1}, x_{2}\right)\right)_{i} +  
      \epsilon^{2}D_{Q}^{i}\left(x_{1}, x_{2}\right)   + O\left(\epsilon^{3}\vert\left(x_{1}, x_{2}\right)\vert^{3}\right),
 \label{exp1}\eern
where $D_{Q}^{i}$ are quadratic forms in $\left(x_{1}, x_{2}\right)$ given by
\be       D_{Q}^{i}\left(x_{1}, x_{2}\right) = \frac{1}{2}\sum_{j,k=1}^{2}\frac{\partial^{3}\psi}{\partial x_{i}\partial x_{j}\partial x_{k}}\left(0, 0\right)x_{j}x_{k}. \nonumber \ee
Concerning the outer normal $\nu$, we have also
\bern 
   \nu = \frac{\left(\frac{\partial\psi_{\epsilon}}{\partial x_{1}}, \frac{\partial\psi_{\epsilon}}{\partial x_{2}}, -1\right) }{\sqrt{1 + \vert\nabla\psi_{\epsilon}\vert^{2}}} 
       = \left(\epsilon\left(A_{Q}\left(x_{1}, x_{2}\right)\right) +   \epsilon^{2}D_{Q}\left(x_{1}, x_{2}\right), -1 + \frac{1}{2}\epsilon^{2}\vert A_{Q}\left(x_{1}, x_{2}\right)\vert^{2}\right)  \nonumber\\
         +     O\left(\epsilon^{3}\vert\left(x_{1}, x_{2}\right)\vert^{3}\right). 
\label{nu}\eern

Now we introduce a new set of coordinates on $B_{\frac{\mu_{1}}{\epsilon}}\left(Q\right) \bigcap\tilde{\Omega}_{\epsilon}$: 
\be       z_{1}=x_{1},  \qquad  z_{2}=x_{2},        \qquad     z_{3} =  x_{3} - \psi_{\epsilon}\left(x_{1}, x_{2}\right). \nonumber\ee
The advantage of these coordinates is that $\partial\tilde{\Omega}_{\epsilon}$ identifies with $\left\lbrace z_{3}=0\right\rbrace$, 
but, as before, 
the corresponding metric $\tilde{g}=\left(\tilde{g}_{ij}\right)_{ij}$ 
will not be flat anymore. 
Its coefficients are given by
\bern \left(\tilde{g}_{ij}\right) = \left(\frac{\partial x}{\partial z_{i}}\cdot\frac{\partial x}{\partial z_{j}}\right) =
\left( 
\begin{array}{ccc} 
    1 +  \frac{\partial\psi_{\epsilon}}{\partial z_{1}}\frac{\partial\psi_{\epsilon}}{\partial z_{1}}  
     & \frac{\partial\psi_{\epsilon}}{\partial z_{1}}\frac{\partial\psi_{\epsilon}}{\partial z_{2}} 
     & \frac{\partial\psi_{\epsilon}}{\partial z_{1}}    \\ 
    \frac{\partial\psi_{\epsilon}}{\partial z_{2}}\frac{\partial\psi_{\epsilon}}{\partial z_{1}} 
    &  1 +  \frac{\partial\psi_{\epsilon}}{\partial z_{2}}\frac{\partial\psi_{\epsilon}}{\partial z_{2}}      
    &  \frac{\partial\psi_{\epsilon}}{\partial z_{2}}      \\
      \frac{\partial\psi_{\epsilon}}{\partial z_{1}}     &   \frac{\partial\psi_{\epsilon}}{\partial z_{2}}  & 1
\end{array} 
\right).
\nonumber\eern
From the estimates in $(\ref{exp1})$ it follows that
\be       \tilde{g}_{ij} = Id + \epsilon A + \epsilon^{2}B + O\left(\epsilon^{3}\vert\left(z_{1}, z_{2}\right)\vert^{3}\right),      \label{gij1}   \ee
where
\be
 A =
\left( 
\begin{array}{cc} 
    0 &  A_{Q}\left(z_{1}, z_{2}\right)   \\ 
    \left(A_{Q}\left(z_{1}, z_{2}\right)\right)^{T}   & 0 
\end{array} 
\right), \nonumber\ee
and
\be
B =
\left( 
\begin{array}{cc} 
     A_{Q}\left(z_{1}, z_{2}\right)\otimes A_{Q}\left(z_{1}, z_{2}\right)  & D_{Q}\left(z_{1}, z_{2}\right)\\ 
    \left(D_{Q}\left(z_{1}, z_{2}\right)\right)^{T}   & 0 
\end{array} 
\right).\footnote{If the vector $v$ has components $\left(v_{i}\right)_{i}$, the notation $v\otimes v$ 
denotes the square matrix with entries $\left(v_{i}v_{j}\right)_{ij}$.} \nonumber\ee
It is also easy to check that the inverse matrix $\left(\tilde{g}^{ij}\right)$ 
is of the form $\tilde{g}^{ij} = Id - \epsilon A + \epsilon^{2}C + O\left(\epsilon^{3}\vert\left(z_{1}, z_{2}\right)\vert^{3}\right)$, 
where
\be
C =
\left( 
\begin{array}{cc} 
     0  & -D_{Q}\left(z_{1}, z_{2}\right)\\ 
    -\left(D_{Q}\left(z_{1}, z_{2}\right)\right)^{T}   & \vert A_{Q}\left(z_{1}, z_{2}\right)\vert^{2}
\end{array} 
\right).  
\nonumber\ee
Furthermore one has $\det\tilde{g} = 1$. 
Therefore, by $(\ref{gij1})$, for any smooth function $u$ there holds
\bern 
    \Delta_{\tilde{g}}u = \Delta u -\epsilon\left[2\left(A_{Q}\left(z_{1}, z_{2}\right)\cdot\nabla_{\left(z_{1}, z_{2}\right)}\frac{\partial u}{\partial z_{3}}\right) 
    + trA_{Q}\frac{\partial u}{\partial z_{3}}\right]   \nonumber\\
  + \epsilon^{2}\left[-2\left(D_{Q}\cdot\nabla_{\left(z_{1}, z_{2}\right)}\frac{\partial u}{\partial z_{3}}\right)  + 
     \vert A_{Q}\left(z_{1}, z_{2}\right)\vert^{2}\frac{\partial^{2} u}{\partial z_{3}\partial z_{3}} -
      divD_{Q}\frac{\partial u}{\partial z_{3}}\right]   \nonumber\\
  + O\left(\epsilon^{3}\vert\left(z_{1}, z_{2}\right)\vert^{3}\right) \vert\nabla^{2}u\vert +  O\left(\epsilon^{3}\vert\left(z_{1}, z_{2}\right)\vert^{3}    \right)\vert\nabla u\vert.
\nonumber\eern
Moreover, from $(\ref{nu})$, we obtain the expression of the unit outer normal to $\partial\tilde{\Omega}_{\epsilon}$, 
$\tilde{\nu}$, in the new coordinates $z$:
\bern 
   \tilde{\nu} =  \left(\epsilon\left(A_{Q}\left(z_{1}, z_{2}\right)\right) +  \epsilon^{2}D_{Q}\left(z_{1}, z_{2}\right), -1 + \frac{3}{2}\epsilon^{2}\vert A_{Q}\left(z_{1}, z_{2}\right)\vert^{2}\right)  \nonumber\\
         +     O\left(\epsilon^{3}\vert\left(z_{1}, z_{2}\right)\vert^{3}\right). 
\nonumber\eern 
Finally the area-element of $\partial\tilde{\Omega}_{\epsilon}$ can be estimated as
\be     d\sigma = \left( 1+  O\left(\epsilon^{2}\vert\left(z_{1}, z_{2}\right)\vert^{2}\right)\right) dz_{1}dz_{2}. \nonumber\ee

Now, locally, in a suitable neighborhood of $Q\in\Gamma$, 
we can consider $\Omega$ as the intersection of two smooth domains 
$\tilde{\Omega}_{1}$ and $\tilde{\Omega}_{2}$ if the opening angle at $Q$ is less than $\pi$, 
or as the union of them if the opening angle is greater than $\pi$. 
In the first case one has 
$\partial\Omega = \left(\partial\tilde{\Omega}_{1}\cap\tilde{\Omega}_{2}\right) \cup \left(\partial\tilde{\Omega}_{2}\cap\tilde{\Omega}_{1}\right)$, 
whereas in the second case 
$\partial\Omega = \left(\partial\tilde{\Omega}_{1}\cap\tilde{\Omega}^{c}_{2}\right) \cup \left(\partial\tilde{\Omega}_{2}\cap\tilde{\Omega}^{c}_{1}\right)$. 
Then, locally, one can straighten $\Gamma$ 
and stretch the two parts of the boundary using the coordinates $z$ for the 
smooth domains $\tilde{\Omega}_{1}$ and $\tilde{\Omega}_{2}$.


\section{Study of the non degeneracy for the unperturbed problem in the cone}

\noindent
Let us consider $K=\tilde{K}\times\R\subset\R^{3}$, 
where $\tilde{K}\subset\R^{2}$ is a cone of opening angle $\alpha$, 
and the problem 
\bern
\left\{ 
\begin{array}{ll} 
         -\Delta u + u = u^ {p}   \quad & \mathrm{in\ } K,     \\
          \frac{\partial u}{\partial \nu} = 0 \qquad  & \mathrm{on\ } \partial K,
 \end{array} 
\right.         
 \label{problem2}\eern
where $p>1$. 
 
\noindent
If $p<5$ and if $u\in W^{1,2}\left(K\right)$, 
solutions of $(\ref{problem2})$ can be found as critical points 
of the functional $I_{K}:W^{1,2}\left(K\right)\rightarrow\R$ defined as
\be  I_{K}\left(u\right)  = \frac{1}{2} \int_{K}\left(\vert\nabla u\vert^ {2} + u^ {2}\right)dx - \frac{1}{p+1}\int_{K}\vert u\vert^ {p+1}dx.
  \label{var2}
\ee
Note that $I_{k}$ is well defined on $W^{1,2}\left(K\right)$; 
in fact, since $K$ is Lipschitz, 
the Sobolev embeddings hold for $p\leq 5$; 
see for instance \cite{Ad}, \cite{Gri}.  

Let us consider also the elliptic equation in $\R^{3}$
\be   -\Delta u + u = u^{p},    \quad u\in W^{1,2}\left(\R^{3}\right),    \quad  u>0,    \label{problem3}\ee
which has a positive radial solution $U$; 
see for instance \cite{AM}, \cite{BL}, \cite{Ma}, \cite{St}. 
It has been shown in \cite{Kw} that such a solution is unique. 
Moreover $U$ and its radial derivatives decay to zero exponentially: 
more precisely satisfy the properties 
\be   \lim_{r\rightarrow +\infty} e^{r} r U\left(r\right) =  c_{3, p},   \qquad   
         \lim_{r\rightarrow +\infty}\frac{U'\left(r\right)}{U\left(r\right)} =  -   \lim_{r\rightarrow +\infty}\frac{U''\left(r\right)}{U\left(r\right)} = -1,
\nonumber\ee 
where $r=\vert x\vert$ and $c_{3, p}$ is a positive constant depending only on the dimension $n=3$ and $p$; 
see \cite{BL}. 

Now, if $p$ is subcritical, the function $U$ is also a solution of problem $(\ref{problem2})$. 
Moreover, if we consider a coordinate system with the $x_{1}$-axis 
coinciding with the edge of $K$, 
the problem $(\ref{problem2})$ is invariant 
under a translation along the $x_{1}$-axis. 
This means that any 
\be     U_{x_{1}}\left(x\right) = U\left(x - \left(x_{1}, 0, 0\right)\right)      \nonumber\ee
is also a solution of $(\ref{problem2})$. 
Then the functional $I_{k}$ has a non-compact critical manifold given by
\be        Z  =  \left\lbrace  U_{x_{1}}\left(x\right) :  x_{1}\in\R\right\rbrace \simeq\R .   \nonumber\ee
Now, to apply the results of the previous section, 
we have to characterize the spectrum and some eigenfunctions of $I_{K}''\left(U_{x_{1}}\right)$. 
More precisely we have to show the following
\begin{lemma} 
Suppose $\alpha\in\left(0, 2\pi\right)\setminus\left\lbrace\pi\right\rbrace$. 
Then the following properties are true:
\begin{itemize}
\item[a)] $T_{U_{x_{1}}}Z=\Ker\left[I_{K}''\left(U_{x_{1}}\right)\right]$, 
for all $x_{1}\in\R$;
\item[b)]$I_{K}''\left(U_{x_{1}}\right)$ is an index $0$ Fredholm map 
\footnote{A linear map $T\in L\left(H,H\right)$ is Fredholm if the kernel is finite-dimensional and the image is closed and has finite codimension. The index of $T$ is $\dim\left(\Ker\left[T\right]\right) - codim\left(Im\left[T\right]\right)$.}  , for all $x_{1}\in\R$.
\end{itemize}   
\label{lem:cono}\end{lemma} 

\begin{remark}
The properties $a)$ and $b)$ imply that $Z$ satisfies condition $ii)$ 
in Subsection $2.1$ and then it is non-degenerate for $I_{K}$.
\end{remark}

\begin{proof} 
We will prove the lemma by taking $x_{1}=0$, hence $U_{0}=U$. 
The case of a general $x_{1}$ will follow immediately. 

Let us show $a)$. 
It is known that there holds the inclusion 
$T_{U}Z\subset\Ker\left[I_{K}''\left(U\right)\right]$; 
see for instance \cite{AM}, Section $2.2$. 
Then it is sufficient to prove that $\Ker\left[I_{K}''\left(U\right)\right]\subset T_{U}Z$. 
Now, $v\in W^{1,2}\left(K\right)$ belongs to $\Ker\left[I_{K}''\left(U\right)\right]$ if and only if
\bern
\left\{ 
\begin{array}{ll} 
         -\Delta v + v = pU^ {p-1}v   \quad & \mathrm{in\ } K,      \\
          \frac{\partial v}{\partial \nu} = 0 \qquad  & \mathrm{on\ } \partial K.
 \end{array} 
\right.         
\label{ker}\eern 
We use the polar coordinates in $K$, $r$, $\theta$, $\varphi$, 
where $r\geq 0$, $0\leq\theta\leq\pi$ and $0\leq\varphi\leq\alpha$. 
Then we write $v\in W^{1,2}\left(K\right)$ in the form
\be        v\left(x_{1}, x_{2}, x_{3}\right) = \sum_{k=0}^{\infty} v_{k}\left(r\right) Y_{k}\left(\theta, \varphi\right),    \label{scomp}\ee
where the $Y_{k}\left(\theta, \varphi\right)$ are the spherical functions satisfying
\bern
\left\{ 
\begin{array}{ll} 
         -\Delta_{S^{2}} Y_{k}  = \lambda_{k}Y_{k}  \quad & \mathrm{in\ } K,      \\
          \frac{\partial Y_{k}}{\partial \varphi} = 0 \qquad  &  \varphi =0, \alpha.
 \end{array} 
\right.         
\label{Yk}\eern
Here $\Delta_{S^{2}}$ denotes the Laplace-Beltrami operator 
on $S^{2}$ (acting on the variables $\theta$, $\varphi$). 
To determine $\lambda_{k}$ and the expression of $Y_{k}$, 
let us split $Y_{k}$ as 
\be        Y_{k}\left(\theta, \varphi\right) = \sum_{m=0}^{\infty} \Theta_{k,m}\left(\theta\right) \Phi_{k,m}\left(\varphi\right)  \nonumber\ee
so that 
\bern 
     \Delta_{S^{2}} Y_{k}    &=&   \sum_{m=0}^{\infty} \left[\frac{1}{\sin\theta} \frac{\partial}{\partial\theta} 
                \left(\sin\theta\frac{\partial}{\partial\theta} \right)  +  \frac{1}{\sin^{2}\theta} \frac{\partial^{2}}{\partial\varphi^{2}}\right] 
                 \Theta_{k,m} \Phi_{k,m}        \nonumber\\
                                       &=&   \sum_{m=0}^{\infty} \left[\frac{1}{\sin\theta}  \frac{d}{d\theta}\left(\sin\theta\Theta '_{k,m}\right) \Phi_{k,m} 
                  + \frac{1}{\sin^{2}\theta}  \Theta_{k,m}\Phi ''_{k,m}\right].
\nonumber\eern
Then $(\ref{Yk})$ becomes
\bern
\left\{ 
\begin{array}{ll} 
         -\sum_{m=0}^{\infty} \left[\frac{1}{\sin\theta}  \frac{d}{d\theta}\left(\sin\theta\Theta '_{k,m}\right) \Phi_{k,m} 
                  + \frac{1}{\sin^{2}\theta}  \Theta_{k,m}\Phi ''_{k,m}\right]   
                  =  \sum_{m=0}^{\infty}\lambda_{k, m}\Theta_{k,m}\Phi_{k,m}     \quad & \mathrm{in\ } K,     \\
           \Phi '_{k,m}\left(0\right) =\Phi '_{k,m}\left(\alpha\right)=0.
 \end{array} 
\right.         
\label{Yk1}\eern
If we require that for all $m$
\bern
\left\{ 
\begin{array}{l} 
         -\Phi ''_{k,m}  = \mu_{m}\Phi_{k,m}  \quad  \mathrm{in\ } \left[0, \alpha\right],    \\
          \Phi '_{k,m}\left(0\right) =\Phi '_{k,m}\left(\alpha\right)=0,
 \end{array} 
\right.         
\label{fi}\eern
we obtain that $\Phi_{k,m}\left(\varphi\right)=a_{k, m} \cos\left(\frac{\pi m}{\alpha}\varphi\right)$ 
satisfies $(\ref{fi})$ with $\mu_{m}=\frac{\pi^{2} m^{2}}{\alpha^{2}}$. 
Replacing this expression in $(\ref{Yk1})$ we have
\bern
\left\{ 
\begin{array}{ll} 
      \sum_{m=0}^{\infty} \left[-\frac{1}{\sin\theta}  \frac{d}{d\theta}\left(\sin\theta\Theta '_{k,m}\right) 
                  + \frac{1}{\sin^{2}\theta} \frac{\pi^{2} m^{2}}{\alpha^{2}} \Theta_{k,m}\right]\Phi_{k,m}   
                  =  \sum_{m=0}^{\infty}\lambda_{k, m}\Theta_{k,m}\Phi_{k,m}   \quad & \mathrm{in\ } K,  \\
            \Phi '_{k,m}\left(0\right) =\Phi '_{k,m}\left(\alpha\right)=0.
 \end{array} 
\right.                              
\nonumber \eern
Since the $\Phi_{k,m}$ are independent, 
we have to solve, for every $m$, the Sturm-Liouville equation
\be     \frac{1}{\sin\theta}  \frac{d}{d\theta}\left(\sin\theta\Theta '_{k,m}\right)   
           +  \left[ \lambda_{k, m} - \frac{1}{\sin^{2}\theta} \frac{\pi^{2} m^{2}}{\alpha^{2}}\right] \Theta_{k,m}  =  0. 
\label{sturm}\ee
Let us rewrite  $(\ref{sturm})$ in the following form
\be     -\frac{1}{\sin\theta}  \frac{d}{d\theta}\left(\sin\theta\Theta '_{k,m}\right)   
            + \frac{1}{\sin^{2}\theta} \frac{\pi^{2} m^{2}}{\alpha^{2}}\Theta_{k,m}  =  \lambda_{k}\Theta_{k,m},
\label{sturm1}\ee
so that we have to determine the eigenvalues $\lambda_{k,m}$ and 
the eigenfunctions of the operator 
\be         -\frac{1}{\sin\theta} \frac{d}{d\theta} \left(\sin\theta\Theta'\left(\theta\right)\right)     
                +\frac{1}{\sin^{2}\theta} \frac{\pi^{2}m^{2}}{\alpha^{2}}\Theta\left(\theta\right).
\nonumber\ee
In order to do this, let us consider the case $\alpha=\pi$, 
that is the following equation
 \be     -\frac{1}{\sin\theta}  \frac{d}{d\theta}\left(\sin\theta\Theta'_{k,m}\right)   
           + \frac{1}{\sin^{2}\theta} m^{2} \Theta_{k,m}  =  \lambda_{k,m}\Theta_{k,m}.
\label{sturm2}\ee
Now, for every $m$, $(\ref{sturm2})$ has solution if 
$\lambda_{k,m}=k\left(k+1\right)$, with $k\geq\vert m\vert$, 
and the solutions are the Legendre polynomials 
$\Theta_{k,m}\left(\theta\right)=P_{k,m}\left(\cos\theta\right)$; 
see for instance \cite{Gro}, \cite{Ho}, \cite{Mu}, \cite{Mu1}. 
Then, for a given value of $k$, there are $2k+1$ 
independent solutions of the form $\Theta_{k,m}\left(\theta\right) \Phi_{k,m}\left(\varphi\right)$, 
one for each integer $m$ with $-k\leq m\leq k$. 
Now, by the classical comparison principle, 
if we decrease $\alpha$ the corresponding eigenvalues $\lambda_{k,m}$, 
given by $(\ref{sturm1})$, should increase, 
whereas if we increase $\alpha$ they should decrease; 
see for instance \cite{Cha}. 
More precisely, if $m=0$ the equations $(\ref{sturm1})$ and $(\ref{sturm2})$ are the same, 
therefore the eigenvalues do not change (and they are $0, 2, 6,...$). 
If $m\geq 1$ we cannot give an explicit expression for the $\lambda_{k,m}$ 
for general $\alpha$, but we can use the comparison principle. 
In conclusion, we obtain that each 
$Y_{k}=\sum_{m=0}^{\infty}\Theta_{k,m}\Phi_{k,m}$ satisfies
\be    -\Delta_{S^{2}} Y_{k}  = \lambda_{k,m} Y_{k}.    \label{bel}\ee

\noindent
Now, one has that
\be      \Delta \left(v_{k} Y_{k}\right)  =   
                       \Delta_{r} \left(v_{k}\right)  Y_{k}   +  \frac{1}{r^{2}} v_{k} \Delta_{S^{2}} Y_{k},   \label{rad}\ee
where $\Delta_{r}$ denotes the Laplace operator in radial coordinates, 
that is $\Delta_{r}=\frac{\partial^{2}}{\partial r^{2}}+\frac{2}{r}\frac{\partial}{\partial r}$.
Then, using $(\ref{scomp})$, $(\ref{bel})$ and $(\ref{rad})$, the condition $(\ref{ker})$ becomes
\be       \sum_{k=0}^{\infty} \left[-v''_{k} - \frac{2}{r} v'_{k}  +  v_{k}  +  \frac{\lambda_{k,m}}{r^{2}} v_{k}   -  pU^{p-1} v_{k}\right] Y_{k} =  0. 
\nonumber\ee
Since the $Y_{k}$ are independent, we get the following equations for $v_{k}$:
\be     A_{k,m}\left(v_{k}\right) := -v''_{k} - \frac{2}{r} v'_{k}  +  v_{k}  +  \frac{\lambda_{k,m}}{r^{2}} v_{k}   -  pU^{p-1} v_{k}  =  0, 
          \quad   m=0, 1, 2...,\ k\geq m.         \nonumber\ee

Let us first consider the case $m=0$. 
If $k=0$, we have to find a $v_{0}$ such that 
\be    A_{0,0}\left(v_{0}\right) = -v''_{0} - \frac{2}{r} v'_{0}  +  v_{0}   -  pU^{p-1} v_{0}  =  0.     \nonumber\ee
It has been shown in \cite{Kw}, Lemma $6$, that all the solutions of $A_{0,0}\left(v\right)=0$ 
are unbounded. 
Since we are looking for solutions $v_{0}\in W^{1,2}\left(\R\right)$, 
it follows that $v_{0}=0$. 

For $k=1$ we have to solve 
\be    A_{1,0}\left(v_{1}\right) = -v''_{1} - \frac{2}{r} v'_{1}  +  v_{1}   + \frac{2}{r^{2}} v_{1} -  pU^{p-1} v_{1}  =  0.   \nonumber\ee
Let $\hat{U}\left(r\right)$ denote the function such that 
$U\left(x\right)=\hat{U}\left(\vert x\vert\right)$, where $U\left(x\right)$ is 
the solution of $(\ref{problem3})$. 
Reasoning as in the proof of Lemma $4.1$ in \cite{AM}, 
we obtain that the family of solutions 
of $A_{1,0}\left(v_{1}\right)=0$, 
with $v_{1}\in W^{1,2}\left(\R\right)$, 
is given by $v_{1}\left(r\right)=c\hat{U}'\left(r\right)$, 
for some $c\in\R$. 

Now, let us show that the equation $A_{k,0}\left(v_{k}\right)=0$ 
has only the trivial solution in $W^{1,2}\left(\R\right)$, 
provided that $k\geq 2$. 
First of all, note that the operator $A_{1,0}$ has the solution $\hat{U}'$ 
which does not change sign in $\left(0, \infty\right)$ 
and therefore is a non-negative operator. 
In fact, if $\sigma$ denotes its smallest eigenvalue, 
any corresponding eigenfunction $\psi_{\sigma}$ 
does not change sign. 
If $\sigma<0$, then $\psi_{\sigma}$ should be orthogonal 
to $\hat{U}'$ and this is a contradiction. 
Thus $\sigma\geq 0$ and $A_{1,0}$ is non-negative. 
Now, we can write
\be        A_{k,0}   =    A_{1,0}      +     \frac{\lambda_{k,0}-2}{r^{2}}.     \nonumber\ee
Since $\lambda_{k,0}-2>0$ whenever $k\geq 2$, 
it follows that $A_{k,0}$ is a positive operator. 
Thus $A_{k,0}\left(v_{k}\right)=0$ implies that $v_{k}=0$. 

If $m\geq 1$ and $\alpha<\pi$, 
using the comparison principle, 
we obtain that each $\lambda_{k, m}$ is greater than $2$. 
Then, reasoning as above, 
we have that each $v_{k}=0$. 

Let us consider the case $\alpha>\pi$. 
If $m=1$ and $k=1$, 
using again the comparison principle, 
we have that $0<\lambda_{1, 1}<2$; 
whereas for $m=1$, $k\geq 2$, 
and for $m\geq 2$, $k\geq m$, 
we have that each $\lambda_{k, m}>2$. 
Then in the last two cases we can use 
the non-negativity of the operator $A_{1,0}$ 
and conclude that $v_{k}=0$. 
In the case $m=1$ and $k=1$ we note that 
the operator 
\be      A_{1, 1}\left(v_{1}\right) := -v''_{1} - \frac{2}{r} v'_{1}  +  v_{1}  +  \frac{\lambda_{1, 1}}{r^{2}} v_{1}   -  pU^{p-1} v_{1}  \nonumber\ee
has a negative eigenvalue, instead of the eigenvalue $0$, since $\lambda_{1, 1}<2$. 
Then also $v_{1}=0$. 

Putting together all the previous information, 
we deduce that any $v\in\Ker\left[I''\left(U\right)\right]$ 
has to be of the form
\be      v\left(x_{1}, x_{2}, x_{3}\right)    =  c\hat{U}'\left(r\right)  Y_{1}\left(\theta, \varphi\right).          
\nonumber\ee   
Now, $Y_{1}$ is such that $-\Delta_{S^{2}} Y_{1} =\lambda_{1,m}Y_{1}$, 
namely it belongs to the kernel of the operator $-\Delta_{S^{2}}-\lambda_{1,m}Id$, 
and such a kernel is 1-dimensional. 
In conclusion, we find that 
\be     v\in span\left\lbrace \hat{U}' Y_{1} \right\rbrace  =  span\left\lbrace \frac{\partial U}{\partial x_{1}} \right\rbrace  = T_{U}Z.    \nonumber\ee
This proves that $a)$ holds. 
It is also easy to check that the operator $I_{K}''\left(U\right)$ 
is a compact perturbation of the identity, 
showing that $b)$ holds true, too. 
This complete the proof of Lemma $\ref{lem:cono}$.
\end{proof}

\begin{remark}
Since $U$ is a Mountain-Pass solution of $(\ref{problem3})$, 
the spectrum of $I_{K}''\left(U\right)$ has one negative simple eigenvalue, $1-p$, 
with eigenspace spanned by $U$ itself. 
Moreover, we have shown in the preceding lemma that $\lambda=0$ 
is an eigenvalue with multiplicity $1$ and 
eigenspace spanned by $\frac{\partial U}{\partial x_{1}}$. 
If $\alpha<\pi$ the rest of the spectrum is positive. 
Whereas if $\alpha>\pi$ there is an other negative simple eigenvalue, 
corresponding to an eigenfunction $\tilde{U}$ given by
\be   \tilde{U}\left(r, \theta, \varphi\right) = \tilde{u}\left(r\right) \cos\left(\frac{\pi}{\alpha}\varphi\right) \tilde{\Theta}\left(\theta\right), \nonumber\ee
where $\tilde{\Theta}$ satisfies $(\ref{sturm})$ with $m=1$ and $k=1$, 
and $\tilde{u}$ satisfies the equation
\be    -v'' - \frac{2}{r} v'  +  v  +  \frac{\lambda_{1,1}}{r^{2}} v   -  pU^{p-1} v  =  0.   \label{utilde}\ee
From $(\ref{utilde})$ one has that there exists a positive constant $C$ 
such that, for $r$ sufficiently large, $\tilde{u}\left(r\right)\leq C e^{-r/C}$. 
In conclusion, one has the following result:
\label{eigenvalues}\end{remark}

\begin{corollary} 
Let $U$ and $\tilde{U}$ be as above and consider the functional $I_{K}$ given in $(\ref{var2})$. 
Then for every $x_{1}\in\R$, $U_{x_{1}}\left(x\right) = U\left(x - \left(x_{1}, 0, 0\right)\right)$ 
is a critical point of $I_{K}$. 
Moreover, the kernel of $I_{K}''\left(U\right)$ is generated by $\frac{\partial U}{\partial x_{1}}$. 
If $\alpha<\pi$ the operator has only one negative eigenvalue, 
and therefore there exists $\delta>0$ such that
\be      I_{K}''\left(U\right)\left[v, v\right]  \geq \delta \norm{v}^{2},   \quad \mathrm{for\ every\ }      v\in W^{1,2}\left(K\right),     v\bot U, \frac{\partial U}{\partial x_{1}}. \nonumber\ee  
If $\alpha>\pi$ the operator has two negative eigenvalues, 
and therefore there exists $\delta>0$ such that 
\be      I_{K}''\left(U\right)\left[v, v\right]  \geq \delta \norm{v}^{2},   \quad \mathrm{for\ every\ }      
v\in W^{1,2}\left(K\right),     v\bot U, \tilde{U}, \frac{\partial U}{\partial x_{1}}. \nonumber\ee  
\label{cor:nondeg}\end{corollary}

\section{Proof of Theorem \ref{th:solution}} 

For every $Q$ on the edge $\Gamma$ of $\partial\Omega_{\epsilon}$, 
let $\mu=\min\left\lbrace \mu_{i}\right\rbrace$, 
so that in $B_{\frac{\mu}{\epsilon}}\left(Q\right) \bigcap\Omega_{\epsilon}$ 
we can use the new set of coordinates $z$. 
Now we choose a cut-off function $\varphi_{\mu}$ with the following properties
\bern
\left\{ 
\begin{array}{ll} 
           \varphi_{\mu}\left(x\right) = 1 \qquad & \mathrm{in\ }   B_{\frac{\mu}{4}}\left(Q\right),    \\
            \varphi_{\mu}\left(x\right) = 0 \qquad  & \mathrm{in\ } \R^{3}\setminus B_{\frac{\mu}{2}}\left(Q\right), \\
            \vert\nabla\varphi_{\mu}\vert + \vert\nabla^{2}\varphi_{\mu}\vert\leq C \qquad  & \mathrm{in\ } B_{\frac{\mu}{2}}\left(Q\right)\setminus B_{\frac{\mu}{4}}\left(Q\right).
 \end{array} 
\right.         
 \label{cutoff}\eern 
For any $Q\in\Gamma$, we define the following function, 
in the coordinates $\left(z_{1}, z_{2}, z_{3}\right)$,
\be U_{Q,\epsilon}\left(z\right)  :=  
     \varphi_{\mu}\left(\epsilon z\right)U_{Q}\left(z\right), 
\label{function}\ee
where $U_{Q}\left(z\right)=U\left(z-Q\right)$. 
Then we consider the manifold
\be     Z_{\epsilon} = \left\lbrace U_{Q,\epsilon}   :    Q\in\Gamma\right\rbrace.     \nonumber\ee
Now, we estimate the gradient of $I_{\epsilon}$ at $U_{Q,\epsilon}$, 
showing that $Z_{\epsilon}$ constitute a manifold of pseudo-critical points of $I_{\epsilon}$.

\begin{lemma}
There exists $C>0$ such that for $\epsilon$ small there holds
\be \norm{I'_{\epsilon}\left(U_{Q,\epsilon}\right)} \leq C\epsilon,   \qquad \mathrm{for\ all\ } Q\in\Gamma. \nonumber\ee
\label{lem:pseudo}\end{lemma}

\begin{proof} 
Let $v\in W^{1,2}\left(\Omega_{\epsilon}\right)$. 
Since the function $U_{Q,\epsilon}$ is supported in $B:=B_{\frac{\mu}{2\epsilon}}\left(Q\right)$, see $(\ref{function})$, 
we can use the coordinate $z$ in this set, and we obtain
\bern
    I'_{\epsilon}\left(U_{Q,\epsilon}\right)\left[v\right]  
    &=& \int_{\partial\Omega_{\epsilon}}\frac{\partial U_{Q,\epsilon}}{\partial\tilde{\nu}} v d\tilde{\sigma} 
    + \int_{\Omega_{\epsilon}}\left(-\Delta_{\tilde{g}}U_{Q,\epsilon} + U_{Q,\epsilon} - \vert U_{Q,\epsilon}\vert^{p}\right) v dV_{\tilde{g}}\left(z\right) \nonumber\\
    &\doteqdot& I + II.
\nonumber\eern
Let us now estimate $I$: 
\be
I =  \int_{\partial\Omega_{\epsilon 1}}\frac{\partial U_{Q,\epsilon}}{\partial\tilde{\nu_{1}}} v d\tilde{\sigma_{1}} + 
       \int_{\partial\Omega_{\epsilon 2}}\frac{\partial U_{Q,\epsilon}}{\partial\tilde{\nu_{2}}} v d\tilde{\sigma_{2}}  \doteqdot 
       I_{1} + I_{2}.
\nonumber\ee
If $K=K_{\alpha\left(Q\right)}$ denotes the cone of angle equal to the angle of the edge in $Q$, we have
\bern
 I_{1} =  \int_{\partial K} \left(U_{Q}\left(z\right)\nabla\varphi_{\mu}\left(\epsilon z\right)\cdot\tilde{\nu_{1}}   
               + \varphi_{\mu}\left(\epsilon z\right)\nabla U_{Q}\left(z\right)\cdot\tilde{\nu_{1}}\right)  v d\tilde{\sigma_{1}} \nonumber\\
         = \int_{\partial K}  U_{Q}\left(z\right)\nabla\varphi_{\mu}\left(\epsilon z\right)\cdot\left(\epsilon\left(A_{Q}\left(z_{1}, z_{2}\right)\right) +  \epsilon^{2}D_{Q}\left(z_{1}, z_{2}\right), -1 + \frac{3}{2}\epsilon^{2}\vert A_{Q}\left(z_{1}, z_{2}\right)\vert^{2}\right)    \nonumber\\
               + \varphi_{\mu}\left(\epsilon z\right)\nabla U_{Q}\left(z\right)\cdot\left(\epsilon\left(A_{Q}\left(z_{1}, z_{2}\right)\right) +  \epsilon^{2}D_{Q}\left(z_{1}, z_{2}\right), -1 + \frac{3}{2}\epsilon^{2}\vert A_{Q}\left(z_{1}, z_{2}\right)\vert^{2}\right)  \nonumber\\
               v \left( 1+  O\left(\epsilon^{2}\vert\left(z_{1}, z_{2}\right)\vert^{2}\right)\right) dz_{1}dz_{2} \nonumber\\
       \doteqdot  a + b. 
\nonumber\eern
Since $\nabla\varphi_{\mu}\left(\epsilon\cdot\right)$ is supported in 
$\R^{3}\setminus B_{\frac{\mu}{4\epsilon}}\left(Q\right)$ and $U_{Q}$ has an exponential decay, 
we have that, for $\epsilon$ small,
\be \vert a\vert \leq C\epsilon e^{-\frac{\mu}{4\epsilon}} \int_{\partial K} \vert v\vert dz_{1}dz_{2}. \label{stima4}\ee
On the other hand
\bern 
     b  = \int_{\frac{\mu}{4\epsilon}\leq\vert z-Q\vert\leq\frac{\mu}{2\epsilon}} 
     \varphi_{\mu}\left(\epsilon z\right)\nabla U_{Q}\left(z\right)\cdot\left(\epsilon\left(A_{Q}\left(z_{1}, z_{2}\right)\right) +  
     \epsilon^{2}D_{Q}\left(z_{1}, z_{2}\right), -1 + \frac{3}{2}\epsilon^{2}\vert A_{Q}\left(z_{1}, z_{2}\right)\vert^{2}\right)  \nonumber\\
               v \left( 1+  O\left(\epsilon^{2}\vert\left(z_{1}, z_{2}\right)\vert^{2}\right)\right) dz_{1}dz_{2} \nonumber\\
           + \int_{\vert z-Q\vert\leq\frac{\mu}{4\epsilon}} 
     \varphi_{\mu}\left(\epsilon z\right)\nabla U_{Q}\left(z\right)\cdot\left(\epsilon\left(A_{Q}\left(z_{1}, z_{2}\right)\right) +  
     \epsilon^{2}D_{Q}\left(z_{1}, z_{2}\right), -1 + \frac{3}{2}\epsilon^{2}\vert A_{Q}\left(z_{1}, z_{2}\right)\vert^{2}\right)  \nonumber\\
               v \left( 1+  O\left(\epsilon^{2}\vert\left(y_{1}, y_{2}\right)\vert^{2}\right)\right) dy_{1}dy_{2} \nonumber\\
         \leq  C\epsilon e^{-\frac{\mu}{4\epsilon}} \int_{\partial K} \vert v\vert dz_{1}dz_{2} + 
            C\epsilon \int_{\partial K} \vert\nabla U_{Q}\vert\cdot\vert v\vert dz_{1}dz_{2}. 
\label{stima5} \eern
The estimates $(\ref{stima4})$ and $(\ref{stima5})$, 
and the trace Sobolev inequalities imply 
$\vert I_{1} \vert \leq C\epsilon \norm {v}$. 
In the same way we can estimate $I_{2}$, 
getting
\be \vert I\vert \leq C\epsilon \norm {v}. \label{stima6}\ee 
Now let's evaluate $II$. 
Using $(\ref{lapl})$ one has
\bern
   II =  \int_{K}\left(-\Delta U_{Q,\epsilon} + U_{Q,\epsilon} - \vert U_{Q,\epsilon}\vert^{p}\right) v dV_{\tilde{g}}\left(z\right) \nonumber\\
           + \epsilon\int_{K}\left[2\left(\gamma''\left(0\right)z_{1}\cdot\nabla_{\left( z_{2}, z_{3}\right) }\frac{\partial  U_{Q,\epsilon}}{\partial z_{1}}\right) + \left(\gamma''\left(0\right)\cdot\nabla_{\left(z_{2}, z_{3}\right)} U_{Q,\epsilon}\right)\right] v dV_{\tilde{g}}\left(z\right) \nonumber\\
  + O\left(\epsilon^{2}\right) \int_{K}\left(\vert z_{1}\vert^{2}\vert\nabla^{2} U_{Q,\epsilon}\vert + \vert z_{1}\vert^{2}\vert\nabla U_{Q,\epsilon}\vert\right) v dV_{\tilde{g}}\left(z\right)  \nonumber\\
     \doteqdot II_{1} + \epsilon II_{2} + O\left(\epsilon^{2}\right)  II_{3}.
\nonumber\eern 
Since $\Delta U_{Q,\epsilon} = U_{Q}\Delta\varphi_{\mu}\left(\epsilon z\right) + 2 \nabla U_{Q}\cdot\nabla\varphi_{\mu}\left(\epsilon z\right) + \varphi_{\mu}\left(\epsilon z\right)\Delta U_{Q}$ and both $\Delta\varphi_{\mu}\left(\epsilon\cdot\right)$ and 
$\nabla\varphi_{\mu}\left(\epsilon\cdot\right)$ are supported in $\R^{3}\setminus B_{\frac{\mu}{4\epsilon}}\left(Q\right)$, 
we get
\bern
     II_{1} =    \int_{\frac{\mu}{4\epsilon}\leq\vert z-Q\vert\leq\frac{\mu}{2\epsilon}} \left(-U_{Q}\Delta\varphi_{\mu}\left(\epsilon z\right) 
                                     - 2 \nabla U_{Q}\cdot\nabla\varphi_{\mu}\left(\epsilon z\right)\right) v\left(1 + O\left(\epsilon \vert z\vert\right)\right)dz 
                                     \nonumber\\
                    +   \int_{\frac{\mu}{4\epsilon}\leq\vert z-Q\vert\leq\frac{\mu}{2\epsilon}}   \left(-\varphi_{\mu}\left(\epsilon z\right)\Delta U_{Q} + 
                                     U_{Q,\epsilon} - \vert U_{Q,\epsilon}\vert^{p}\right) v\left(1 + O\left(\epsilon \vert z\vert\right)\right)dz 
                                     \nonumber\\
                    +    \int_{\vert z-Q\vert\leq\frac{\mu}{4\epsilon}}   \left(-\Delta U_{Q} + 
                                     U_{Q} - \vert U_{Q}\vert^{p}\right) v\left(1 + O\left(\epsilon \vert z\vert\right)\right)dz.      \label{II2}
\eern
Since $U_{Q}$ is a solution in $\R^{3}$ the last term in $(\ref{II2})$ vanishes, 
and using the exponential decay of $U_{Q}$ at infinity and the properties of the cut-off function, see $(\ref{cutoff})$, 
one has
\be        \vert   II_{1}\vert \leq C e^{-\frac{\mu}{4\epsilon}} \int_{K} \vert v\vert dz.   \nonumber\ee
By $(\ref{function})$ we can compute also 
$\nabla_{\left(z_{2}, z_{3}\right) }\frac{\partial  U_{Q,\epsilon}}{\partial z_{1}}$ 
and $\nabla_{\left(z_{2}, z_{3}\right)} U_{Q,\epsilon}$ and we have
\bern
      II_{2} =  \int_{K}  2\gamma''\left(0\right)z_{1}\cdot\left[\nabla_{\left(z_{2}, z_{3}\right)}\frac{\partial\varphi_{\mu}\left(\epsilon z\right)}{\partial z_{1}}U_{Q}  
      +  \nabla_{\left(z_{2}, z_{3}\right)} \varphi_{\mu}\left(\epsilon z\right) \frac{\partial U_{Q}}{\partial z_{1}} \right]     \nonumber\\
      + 2\gamma''\left(0\right)z_{1}\cdot\left[\frac{\partial\varphi_{\mu}\left(\epsilon z\right)}{\partial z_{1}}  \nabla_{\left(z_{2}, z_{3}\right)}  U_{Q}     
      +    \varphi_{\mu}\left(\epsilon z\right) \nabla_{\left(z_{2}, z_{3}\right)}\frac{\partial U_{Q}}{\partial z_{1}}\right] \nonumber\\
      + \gamma''\left(0\right)\cdot\left[\nabla_{\left(z_{2}, z_{3}\right)}\varphi_{\mu}\left(\epsilon z\right) U_{Q} 
      +   \varphi_{\mu}\left(\epsilon z\right)\nabla_{\left(z_{2}, z_{3}\right)}U_{Q}\right]  v dV_{\tilde{g}}\left(z\right)  \nonumber\\
      =  \int_{\frac{\mu}{4\epsilon}\leq\vert z-Q\vert\leq\frac{\mu}{2\epsilon}}  
               2\gamma''\left(0\right)z_{1}\cdot\left[\nabla_{\left(z_{2}, z_{3}\right)}\frac{\partial\varphi_{\mu}\left(\epsilon z\right)}{\partial z_{1}}U_{Q}  
              +  \nabla_{\left(z_{2}, z_{3}\right)} \varphi_{\mu}\left(\epsilon z\right) \frac{\partial U_{Q}}{\partial z_{1}}  
              +  \frac{\partial\varphi_{\mu}\left(\epsilon z\right)}{\partial z_{1}}  \nabla_{\left(z_{2}, z_{3}\right)}  U_{Q}  \right]  \nonumber\\
              +   \gamma''\left(0\right)\cdot\nabla_{\left(z_{2}, z_{3}\right)}\varphi_{\mu}\left(\epsilon z\right) U_{Q}   v dV_{\tilde{g}}\left(z\right)  \nonumber\\
              +  \int_{\vert z-Q\vert\leq\frac{\mu}{2\epsilon}}  
                    \varphi_{\mu}\left(\epsilon z\right) \left[ 2\gamma''\left(0\right)z_{1}\cdot\nabla_{\left(z_{2}, z_{3}\right)}\frac{\partial U_{Q}}{\partial z_{1}} 
                    +  \gamma''\left(0\right)\cdot\nabla_{\left(z_{2}, z_{3}\right)}U_{Q} \right] v  dV_{\tilde{g}}\left(z\right). \nonumber
\eern 
Hence
\bern
   \vert  II_{2} \vert \leq 
                       C \int_{\frac{\mu}{4\epsilon}\leq\vert z-Q\vert\leq\frac{\mu}{2\epsilon}} 
                             \left[ 2\vert\gamma''\left(0\right)\vert\cdot\vert z_{1}\vert \left( \vert U_{Q}\vert +  
                                      \vert \frac{\partial U_{Q}}{\partial z_{1}}\vert  + \vert\nabla_{\left(z_{2}, z_{3}\right)} U_{Q}\vert\right)  
                                  + \vert\gamma''\left(0\right)\vert \cdot\vert U_{Q}\vert\right] \vert v\vert dV_{\tilde{g}}\left(z\right)  \nonumber\\
                       + \int_{\vert z-Q\vert\leq\frac{\mu}{2\epsilon}}  2 \vert \varphi_{\mu}\left(\epsilon z\right)\vert \cdot  \sup_{Q}\vert\gamma''\left(0\right)\vert 
                               \left( \vert z_{1}\vert \cdot\vert\nabla_{\left(z_{2}, z_{3}\right)}\frac{\partial U_{Q}}{\partial z_{1}}\vert 
                                       + \vert \nabla_{\left(z_{2}, z_{3}\right)}U_{Q} \vert\right) \vert v\vert dV_{\tilde{g}}\left(z\right). \nonumber
\eern 
Using again the exponential decay of $U_{Q}$ at infinity 
one can estimate the first term by $C e^{-\frac{\mu}{4\epsilon}} \int_{K} \vert v\vert dz$ 
and conclude that the second term is bounded. 
In the same way we can estimate $II_{3}$, getting
\be     \vert II \vert  \leq C\epsilon \norm {v}. \label{sti}\ee
From   $(\ref{stima6})$ and $(\ref{sti})$ we obtain the conclusion.       
\end{proof}

\medskip

Now, we need a result of non-degeneracy, 
which allows us to say that the operator $I''_{\epsilon}\left(U_{Q,\epsilon}\right)$ 
is invertible on the orthogonal complement of $T_{U_{Q,\epsilon}}Z_{\epsilon}$.

\begin{lemma}
There exists $\bar{\delta}>0$ such that for $\epsilon$ small, if $\alpha<\pi$, there holds
\be  I''_{\epsilon}\left(U_{Q,\epsilon}\right)\left[v, v\right]  \geq \bar{\delta} \norm{v}^{2},   \qquad \mathrm{for\ every\ }      v\in W^{1,2}\left(\Omega_{\epsilon}\right),     v\bot U_{Q, \epsilon}, \frac{\partial U_{Q, \epsilon}}{\partial Q}, \nonumber\ee
and, if $\alpha>\pi$, there holds
\be  I''_{\epsilon}\left(U_{Q,\epsilon}\right)\left[v, v\right]  \geq \bar{\delta} \norm{v}^{2},   \qquad \mathrm{for\ every\ }      v\in W^{1,2}\left(\Omega_{\epsilon}\right),     v\bot U_{Q, \epsilon}, \tilde{U}_{Q, \epsilon} \frac{\partial U_{Q, \epsilon}}{\partial Q}, \nonumber\ee
where $\tilde{U}_{Q, \epsilon}$ is defined as $U_{Q, \epsilon}$ in $(\ref{function})$.
\label{lem:nondeg}\end{lemma}

\begin{proof} 
Let us consider the case $\alpha<\pi$. 
Let $R\gg 1$; consider a radial smooth function 
$\chi_{R}:\R^{3}\rightarrow\R$ such that
\bern
\left\{ 
\begin{array}{ll} 
          \chi_{R}\left(x\right) =1  \quad & \mathrm{in\ }   B_{R}\left(0\right),     \\
          \chi_{R}\left(x\right) =0  \quad & \mathrm{in\ }   \R^{3}\setminus B_{2R}\left(0\right),     \\
          \vert\nabla\chi_{R}\vert \leq \frac{2}{R}   \quad & \mathrm{in\ }   B_{2R}\left(0\right)\setminus B_{R}\left(0\right),
 \end{array} 
\right.         
 \label{chi}\eern
and set 
 \be      v_{1}\left(x\right) = \chi_{R}\left(x-Q\right)v\left(x\right),    \qquad  
  v_{2}\left(x\right) = \left(1-\chi_{R}\left(x-Q\right)\right)v\left(x\right).    \nonumber\ee
A straight computation yields
\be   \norm{v}^{2}  =    \norm{v_{1}}^{2}  + \norm{v_{2}}^{2}   +    
                    2\int_{\Omega_{\epsilon}}\left( \nabla v_{1}\cdot\nabla v_{2}  +  v_{1} v_{2}\right)dx.
 \nonumber\ee
 We write $\int_{\Omega_{\epsilon}}\left( \nabla v_{1}\cdot\nabla v_{2}  +  v_{1} v_{2}\right)dx=\gamma_{1} + \gamma_{2}$, where
 \bern
      \gamma_{1}  &=&  \int_{\Omega_{\epsilon}} \chi_{R}\left(1-\chi_{R}\right) \left(v^{2}  +  \vert\nabla v\vert^{2}\right)dx,   \nonumber\\
      \gamma_{2}  &=&  \int_{\Omega_{\epsilon}} \left(v_{2}\nabla v\cdot \nabla\chi_{R}   -   v_{1}\nabla v\cdot \nabla\chi_{R}   
                                  - v^{2}\vert\nabla\chi_{R}\vert^{2}\right)dx.
 \nonumber\eern
Since the integrand in $\gamma_{2}$ is supported in $B_{2R}\left(Q\right)\setminus B_{R}\left(Q\right)$, 
using $(\ref{chi})$ and the Young's inequality 
we obtain that $\vert\gamma_{2}\vert=o_{R}\left(1\right)\norm{v}^{2}$. 
As a consequence we have
\be   \norm{v}^{2}  =    \norm{v_{1}}^{2}  + \norm{v_{2}}^{2}     +  2 \gamma_{1}  + o_{R}\left(1\right)\norm{v}^{2}. 
\nonumber\ee
Now let us evaluate $I''_{\epsilon}\left(U_{Q,\epsilon}\right)\left[v, v\right] = \sigma_{1} + \sigma_{2} + \sigma_{3}$, where
\be           \sigma_{1} = I''_{\epsilon}\left(U_{Q,\epsilon}\right)\left[v_{1}, v_{1}\right],  \quad
                 \sigma_{2} = I''_{\epsilon}\left(U_{Q,\epsilon}\right)\left[v_{2}, v_{2}\right],  \quad 
                 \sigma_{3} = 2I''_{\epsilon}\left(U_{Q,\epsilon}\right)\left[v_{1}, v_{2}\right]. 
\nonumber\ee
Similarly to the previous estimates, since $U_{Q}$ decays exponentially away from $Q$, we get
\bern   
       \sigma_{2} &\geq & C^{-1}\norm{v_{2}}^{2} +  o_{\epsilon, R}\left(1\right) \norm{v}^{2},   \nonumber\\
        \sigma_{3} &\geq & C^{-1}\gamma_{1} +  o_{\epsilon, R}\left(1\right) \norm{v}^{2}. 
\label{sigma23}\eern
Hence it is sufficient to estimate the term $\sigma_{1}$. 
From the exponential decay of $U_{Q}$ and the fact that $v\bot U_{Q, \epsilon}, \frac{\partial U_{Q, \epsilon}}{\partial Q}$ it follows that
\bern
    \left(v_{1}, U_{Q, \epsilon}\right)_{W^{1,2}\left(\Omega_{\epsilon}\right)}  &=& 
            - \left(v_{2}, U_{Q, \epsilon}\right)_{W^{1,2}\left(\Omega_{\epsilon}\right)}   =     
            o_{\epsilon, R}\left(1\right) \norm{v}^{2},    \nonumber\\
      \left(v_{1}, \frac{\partial U_{Q, \epsilon}}{\partial Q}\right)_{W^{1,2}\left(\Omega_{\epsilon}\right)}  &=& 
            - \left(v_{2}, \frac{\partial U_{Q, \epsilon}}{\partial Q}\right)_{W^{1,2}\left(\Omega_{\epsilon}\right)}   =     
            o_{\epsilon, R}\left(1\right) \norm{v}^{2}. 
\label{v1v2}\eern
Moreover, since $U_{Q, \epsilon}$ is supported in $B:=B_{\frac{\mu}{2\epsilon}}\left(Q\right)$, see $(\ref{function})$, 
we can use the coordinate $z$ in this set, and we obtain
\bern 
      \left(v_{1}, U_{Q, \epsilon}\right)_{W^{1,2}\left(\Omega_{\epsilon}\right)}  &=& 
         \int_{\partial\Omega_{\epsilon}}   v_{1}\frac{\partial U_{Q,\epsilon}}{\partial\tilde{\nu}} v d\tilde{\sigma} 
    + \int_{\Omega_{\epsilon}}   v_{1}\left(-\Delta_{\tilde{g}} U_{Q,\epsilon} + U_{Q,\epsilon}\right)   dV_{\tilde{g}}\left(z\right)    \nonumber\\
        &=&  \left(v_{1}, U_{Q}\right)_{W^{1,2}\left(K\right)}     +   
        o_{\epsilon}\left(1\right)\norm{v_{1}},  
\label{v11}\eern
where $K=K_{\alpha}$ is the cone of opening angle equal to the angle of $\Gamma$ in $Q$.
In the same way we can obtain that 
\be       \left(v_{1}, \frac{\partial U_{Q, \epsilon}}{\partial Q}\right)_{W^{1,2}\left(\Omega_{\epsilon}\right)}  = 
              \left(v_{1}, \frac{\partial U_{Q}}{\partial Q}\right)_{W^{1,2}\left(K\right)}     +  
               o_{\epsilon}\left(1\right)\norm{v_{1}}. 
\label{v12}\ee
From the estimates $(\ref{v1v2})$, $(\ref{v11})$ and $(\ref{v12})$, 
we deduce that for $R$ sufficiently large and $\epsilon$ sufficiently small 
\bern
            \left(v_{1}, U_{Q}\right)_{W^{1,2}\left(K\right)} &=& o_{\epsilon, R}\left(1\right)\norm{v_{1}}, \nonumber\\
             \left(v_{1}, \frac{\partial U_{Q}}{\partial Q}\right)_{W^{1,2}\left(K\right)}    &=& o_{\epsilon, R}\left(1\right)\norm{v_{1}}. 
\nonumber\eern
Now we can apply Lemma $\ref{lem:cono}$, getting
 \be      I''\left(U_{Q}\right)\left[v_{1}, v_{1}\right] \geq \delta \norm{v_{1}}_{W^{1,2}\left(K\right)}  
            +  o_{\epsilon, R}\left(1\right).  \nonumber\ee 
Then the following estimate holds
\bern    \sigma_{1}  &=&   I''\left(U_{Q}\right)\left[v_{1}, v_{1}\right]   
                                    +   o_{\epsilon}\left(1\right)\norm{v_{1}} 
                                \geq   \delta \norm{v_{1}}_{W^{1,2}\left(K\right)}  
                                    +  o_{\epsilon, R}\left(1\right) \norm{v} \nonumber\\
                                &\geq &  \delta \norm{v_{1}}
                                    +  o_{\epsilon, R}\left(1\right) \norm{v}.
\label{sigma1}\eern
In conclusion, from $(\ref{sigma23})$ and $(\ref{sigma1})$ we deduce
\be       I''_{\epsilon}\left(U_{Q,\epsilon}\right)\left[v, v\right]  \geq  \delta \norm{v}
               +      o_{\epsilon, R}\left(1\right) \norm{v} 
               \geq  \frac{\delta}{2}\norm{v}, 
\nonumber\ee
provided $R$ is taken large and $\epsilon$ sufficiently small. 
This concludes the proof. 

The case $\alpha>\pi$ has substantially the same proof, 
but we have to consider also the function $\tilde{U}$ and 
use the exponential decay of $\tilde{u}$ at infinity, 
see Remark $\ref{eigenvalues}$. 
\end{proof}

\medskip

The following lemma provides an expansion of the functional 
$I_{\epsilon}\left(U_{Q,\epsilon}\right)$ with respect to $Q$.

\begin{lemma}
For $\epsilon$ small the following expansion holds
\be  I_{\epsilon}\left(U_{Q,\epsilon}\right) = C_{0}\alpha\left(Q\right) + O\left(\epsilon\right),  \label{espansione}\ee
where 
\be  C_{0} = \left(\frac{1}{2}-\frac{1}{p+1}\right) \int_{0}^{\infty}\int_{0}^{\pi}\vert U_{Q}\left(r\right)\vert^{p+1} r \sin^{2}\theta dr d\theta. \nonumber\ee
\label{lem:espansione}\end{lemma}

\begin{proof}  
Since the function $U_{Q,\epsilon}$ is supported in $B:=B_{\frac{\mu}{2\epsilon}}\left(Q\right)$, see $(\ref{function})$, 
we can use the coordinate $z$ in this set, and we obtain
\be
    I_{\epsilon}\left(U_{Q,\epsilon}\right) = 
           \frac{1}{2}\int_{B\cap\Omega_{\epsilon}}\left(\vert\nabla_{\tilde{g}}U_{Q,\epsilon}\vert^{2} + U_{Q,\epsilon}^{2}\right)dV_{\tilde{g}}\left(z\right)  
           - \frac{1}{p+1}\int_{B\cap\Omega_{\epsilon}} \vert U_{Q,\epsilon}\vert^{p+1} dV_{\tilde{g}}\left(z\right). \nonumber
\ee
Integrating by parts, we get
\bern
    I_{\epsilon}\left(U_{Q,\epsilon}\right) = 
           \frac{1}{2}\int_{B\cap\partial\Omega_{\epsilon}} U_{Q,\epsilon}\frac{\partial U_{Q,\epsilon}}{\partial\tilde{\nu}} d\tilde{\sigma} 
            + \frac{1}{2}\int_{B\cap\Omega_{\epsilon}} U_{Q,\epsilon} \left(-\Delta_{\tilde{g}}U_{Q,\epsilon} + U_{Q,\epsilon}\right)dV_{\tilde{g}}\left(z\right)  \nonumber\\
           - \frac{1}{p+1}\int_{B\cap\Omega_{\epsilon}} \vert U_{Q,\epsilon}\vert^{p+1} dV_{\tilde{g}}\left(z\right) \nonumber\\
           \doteqdot I + II, \nonumber
\eern
where $I$ is the surface integral over the boundary and 
$II$ refers to the last two terms. 
Now, $I$ can be split in two terms 
which correspond to the surface integrals on the "faces" of the edge $\Gamma$:
\be I = \frac{1}{2}\int_{B\cap\partial\Omega_{\epsilon 1}} U_{Q,\epsilon}\frac{\partial U_{Q,\epsilon}}{\partial\tilde{\nu_{1}}} d\tilde{\sigma_{1}} 
          + \frac{1}{2}\int_{B\cap\partial\Omega_{\epsilon 2}} U_{Q,\epsilon}\frac{\partial U_{Q,\epsilon}}{\partial\tilde{\nu_{2}}} d\tilde{\sigma_{2}}  
         \doteqdot I_{1} +   I_{2}. 
\nonumber\ee
It is sufficient to evaluate $I_{1}$, since the estimate of $I_{2}$ is similar. 
Using the expression of $U_{Q, \epsilon}$, 
see $(\ref{function})$, we get
\bern
     I_{1} = \frac{1}{2}\int_{B\cap\partial\Omega_{\epsilon 1}} U_{Q,\epsilon}\left(U_{Q}\nabla\varphi_{\mu}\left(\epsilon z\right)   + 
                     \varphi_{\mu}\left(\epsilon z\right)\nabla  U_{Q}\right) \nonumber\\  
      \cdot \left(\epsilon\left(A_{Q}\left(z_{1}, z_{2}\right)\right) +  
     \epsilon^{2}D_{Q}\left(z_{1}, z_{2}\right), -1 + \frac{3}{2}\epsilon^{2}\vert A_{Q}\left(z_{1}, z_{2}\right)\vert^{2}\right) 
     \left( 1+  O\left(\epsilon^{2}\vert\left(z_{1}, z_{2}\right)\vert^{2}\right)\right) dz_{1}dz_{2}    \nonumber\\
              = \frac{1}{2}\int_{\frac{\mu}{4\epsilon}\leq\vert z-Q\vert\leq\frac{\mu}{2\epsilon}}  
                       \varphi_{\mu}\left(\epsilon z\right) U^{2}_{Q}\nabla\varphi_{\mu}\left(\epsilon z\right) \nonumber\\
                       \cdot            \left(\epsilon\left(A_{Q}\left(z_{1}, z_{2}\right)\right) +  
     \epsilon^{2}D_{Q}\left(z_{1}, z_{2}\right), -1 + \frac{3}{2}\epsilon^{2}\vert A_{Q}\left(z_{1}, z_{2}\right)\vert^{2}\right) 
     \left( 1+  O\left(\epsilon^{2}\vert\left(z_{1}, z_{2}\right)\vert^{2}\right)\right) dz_{1}dz_{2}    \nonumber\\
               + \frac{1}{2}\int_{\vert z-Q\vert\leq\frac{\mu}{2\epsilon}} 
                        \varphi^{2}_{\mu}\left(\epsilon z\right) U_{Q}\nabla U_{Q}      \nonumber\\
                        \cdot                         \left(\epsilon\left(A_{Q}\left(z_{1}, z_{2}\right)\right) +  
     \epsilon^{2}D_{Q}\left(z_{1}, z_{2}\right), -1 + \frac{3}{2}\epsilon^{2}\vert A_{Q}\left(z_{1}, z_{2}\right)\vert^{2}\right) 
     \left( 1+  O\left(\epsilon^{2}\vert\left(z_{1}, z_{2}\right)\vert^{2}\right)\right) dz_{1}dz_{2}.    \nonumber
\eern
Similarly to the previous estimates, we get $I_{1} = O\left( e^{-\frac{\mu}{2\epsilon}}\right)  + O\left(\epsilon\right)$. 
Then we obtain that
\be I = O\left(\epsilon\right).       \label{i}\ee
Now, we have to evaluate $II$: 
\bern 
      II =   \frac{1}{2}\int_{B\cap\Omega_{\epsilon}} U_{Q,\epsilon} 
                     \left(-\Delta U_{Q,\epsilon} + U_{Q,\epsilon}\right)\left(1 + O\left(\epsilon\vert z\vert\right)\right) dz  \nonumber\\
                     +  \frac{\epsilon}{2}\int_{B\cap\Omega_{\epsilon}} U_{Q,\epsilon} 
                          \left[2\gamma''\left(0\right)z_{1}\cdot\nabla_{\left(z_{2}, z_{3}\right)}\frac{\partial U_{Q, \epsilon}}{\partial z_{1}} 
                          +  \gamma''\left(0\right)\cdot\nabla_{\left(z_{2}, z_{3}\right)} U_{Q,\epsilon}\right] 
                          \left(1 + O\left(\epsilon\vert z\vert\right)\right) dz  \nonumber\\
                          + O\left(\epsilon^{2}\vert z_{1}\vert^{2}\right)
           - \frac{1}{p+1}\int_{B\cap\Omega_{\epsilon}} \vert U_{Q,\epsilon}\vert^{p+1} \left(1 + O\left(\epsilon\vert z\vert\right)\right) dz.
\nonumber\eern
We have
\be     II =  \left(\frac{1}{2}-\frac{1}{p+1}\right)\alpha\left(Q\right)  \int_{0}^{\infty}\int_{0}^{\pi}\vert U_{Q}\left(r\right)\vert^{p+1} r \sin^{2}\theta dr d\theta   + O\left( \epsilon\right).  \label{ii}\ee
Putting together $(\ref{i})$ and $(\ref{ii})$, we obtain $(\ref{espansione})$ and this concludes the proof. 
\end{proof}

\bigskip

Let $P_{Q}:W^{1,2}\left(\Omega_{\epsilon}\right) \longrightarrow \left(T_{U_{Q, \epsilon}} Z_{\epsilon}\right)^{\perp}$ 
be the projection onto the orthogonal complement of $T_{U_{Q, \epsilon}} Z_{\epsilon}$, 
for all $Q$ on the edge $\Gamma$ of $\partial\Omega_{\epsilon}$. 
According to the lemma $\ref{lem:nondeg}$, 
we have that for $\epsilon$ sufficiently small the operator 
$L_{Q} = P_{Q}\circ I''_{\epsilon}\left(U_{Q,\epsilon}\right)\circ P_{Q}$ 
is invertible and there exists $C>0$ such that 
\be \norm{L_{Q}^{-1}}\leq C.    \nonumber\ee  
Now, using the fact that $I''_{\epsilon}\left(U_{Q,\epsilon}\right)$ is invertible 
on the orthogonal complement of $T_{U_{Q, \epsilon}} Z_{\epsilon}$, 
we will solve the auxiliary equation.
\begin{proposition} 
Let $I_{\epsilon}$ be the functional defined in $(\ref{var1})$. 
Then for $\epsilon>0$ small there exists a unique $w = w\left(\epsilon, Q\right)\in\left(T_{U_{Q, \epsilon}} Z_{\epsilon}\right)^{\perp}$ 
such that $I'_{\epsilon}\left(U_{Q, \epsilon} + w\left(\epsilon, Q\right)\right)\in T_{U_{Q, \epsilon}} Z_{\epsilon}$. 
Moreover the function $w\left(\epsilon, Q\right)$ is of class $C^{1}$ with respect to $Q$ 
and there holds
\be      \norm{w\left(\epsilon, Q\right)} \leq   C\epsilon,   \qquad 
             \norm{\frac{\partial w\left(\epsilon, Q\right)}{\partial Q}} \leq   C\epsilon.    \label{derivw} \ee
\label{prop:w}\end{proposition}

\begin{proof} 
We want to find a solution $w\in\left(T_{U_{Q, \epsilon}} Z_{\epsilon}\right)^{\perp}$ 
of $P_{Q}I'_{\epsilon}\left(U_{Q, \epsilon} + w\right) = 0$. 
For every $w\in\left(T_{U_{Q, \epsilon}} Z_{\epsilon}\right)^{\perp}$ we can write
\be     I'_{\epsilon}\left(U_{Q, \epsilon} + w\right) =  I'_{\epsilon}\left(U_{Q, \epsilon}\right)  
                       + I''_{\epsilon}\left(U_{Q, \epsilon}\right)\left[w\right]     + R_{Q, \epsilon}\left(w\right),   \nonumber\ee
where $R_{Q, \epsilon}\left(w\right)$ is given by
\be   R_{Q, \epsilon}\left(w\right)   =    I'_{\epsilon}\left(U_{Q, \epsilon} + w\right) -  I'_{\epsilon}\left(U_{Q, \epsilon}\right)  
                       - I''_{\epsilon}\left(U_{Q, \epsilon}\right)\left[w\right].         \nonumber\ee
Given $v\in W^{1,2}\left(\Omega_{\epsilon}\right)$ there holds
\be    R_{Q, \epsilon}\left(w\right)\left[v\right] = - \int_{\Omega_{\epsilon}} \left(\vert U_{Q, \epsilon} + w\vert^{p} 
            - \vert U_{Q, \epsilon}\vert^{p}  -  p\vert U_{Q, \epsilon}\vert^{p-1} w\right) v dx.   \nonumber     \ee
Using the following inequality
\bern 
\vert \left(a + b\right)^{p} - a^{p} - p a^{p-1} b \vert \leq
\left\{ 
\begin{array}{ll} 
         C\left(p\right) \vert b\vert^{p}   \quad & \mathrm{for\ } p\leq 2,     \\
         C\left(p\right) \left(\vert b\vert^{2} + \vert b\vert^{p}\right)   \quad & \mathrm{for\ } p> 2,
 \end{array} 
\right.         
 \nonumber\eern 
for $a, b\in\R$, $\vert a\vert\leq 1$, 
the H\"{o}lder's inequality and the Sobolev embeddings we obtain
\be      \norm{R_{Q, \epsilon}\left(w\right)\left[v\right]} 
                 \leq  C \int_{\Omega_{\epsilon}} \left(\vert w\vert^{2} + \vert w\vert^{p}\right) \vert v\vert dx   
                 \leq  C \left(\norm{w}^{2}+ \norm{w}^{p}\right) \norm{v}.       \label{R1} \ee
Similarly, from the inequality
\bern 
\vert \left(a + b_{1}\right)^{p} - \left(a + b_{2}\right)^{p}  - p a^{p-1} \left(b_{1} -b_{2}\right)  \vert    \nonumber\\
\leq
\left\{ 
\begin{array}{ll} 
         C\left(p\right) \left(\vert b_{1}\vert^{p-1} + \vert b_{2}\vert^{p-1}\right) \vert b_{1} -  b_{2}\vert   \quad & \mathrm{for\ } p\leq 2,      \\
         C\left(p\right) \left(\vert b_{1}\vert + \vert b_{2}\vert  +  \vert b_{1}\vert^{p-1} + \vert b_{2}\vert^{p-1}\right) \vert b_{1} -  b_{2}\vert \quad & \mathrm{for\ } p> 2,
 \end{array} 
\right.         
 \nonumber\eern 
for $a, b_{1}, b_{2}\in\R$, $\vert a\vert\leq 1$, 
we get 
\bern      \norm{R_{Q, \epsilon}\left(w_{1}\right)\left[v\right] -  R_{Q, \epsilon}\left(w_{2}\right)\left[v\right]}  
  \leq   C \int_{\Omega_{\epsilon}} \left(\vert w_{1}\vert  +  \vert w_{2}\vert  +  \vert w_{1}\vert^{p-1} + \vert w_{2}\vert^{p-1}\right) 
                 \vert w_{1} - w_{2}\vert \cdot\vert v\vert dx   \nonumber\\
          \leq         C \left(\norm{w_{1}}  +   
                 \norm{w_{2}}   +   \norm{w_{1}}^{p-1}   + 
                 \norm{w_{2}}^{p-1}\right)       
                 \norm{w_{1} - w_{2}} \cdot 
                 \norm{v}.       
\label{R2} \eern
Now, by the invertibility of the operator $L_{Q} = P_{Q}\circ I''_{\epsilon}\left(U_{Q,\epsilon}\right)\circ P_{Q}$, 
we have that the function $w$ solves $P_{Q}I'_{\epsilon}\left(U_{Q, \epsilon} + w\right) = 0$ if and only if 
\be     w = -  \left(L_{Q}\right)^{-1} \left[P_{Q}I'_{\epsilon}\left(U_{Q, \epsilon}\right)  + P_{Q}R_{Q, \epsilon}\left(w\right)\right]. \nonumber\ee
Setting  
\be       N_{Q, \epsilon}\left(w\right) =  
          -  \left(L_{Q}\right)^{-1} \left[P_{Q}I'_{\epsilon}\left(U_{Q, \epsilon}\right)  + P_{Q}R_{Q, \epsilon}\left(w\right)\right], \nonumber\ee
we have to solve
\be     w = N_{Q, \epsilon}\left(w\right).    \nonumber\ee
The norm of $I'_{\epsilon}\left(U_{Q, \epsilon}\right)$ has been estimated in Lemma $\ref{lem:pseudo}$. 
Then from $(\ref{R1})$ and $(\ref{R2})$ we obtain the two relations
\bern
      \norm{N_{Q, \epsilon}\left(w\right)}  &\leq &     C_{1}\epsilon  +  C_{2}\left(\norm{w}^{2}+ \norm{w}^{p}\right),      \label{N1}  \\
      \norm{N_{Q, \epsilon}\left(w_{1}\right)    -   N_{Q, \epsilon}\left(w_{2}\right)}
      &\leq &      
     C \left(\norm{w_{1}} + \norm{w_{2}}   +   
      \norm{w_{1}}^{p-1}  + \norm{w_{2}}^{p-1}\right) 
      \norm{w_{1} - w_{2}}.    \label{N2}
\eern
Now, for $\bar{C}>0$, we define the set
\be       W_{\bar{C}} = \left\lbrace w\in\left(T_{U_{Q, \epsilon}} Z_{\epsilon}\right)^{\perp} : 
                                    \norm{w} \leq \bar{C}\epsilon \right\rbrace. \nonumber\ee
We show that $N_{Q, \epsilon}$ is a contraction in $W_{\bar{C}}$ 
for $\bar{C}$ sufficiently large and for $\epsilon$ small. 
Clearly, by $(\ref{N1})$, if $\bar{C}>2C_{1}$ the set $W_{\bar{C}}$ 
is mapped into itself if $\epsilon$ is sufficiently small. 
Then, if $w_{1}, w_{2}\in W_{\bar{C}}$, by $(\ref{N2})$ there holds
\be   \norm{N_{Q, \epsilon}\left(w_{1}\right)    -   N_{Q, \epsilon}\left(w_{2}\right)}  
            \leq  2C\left(\bar{C}\epsilon   +    \bar{C}^{p-1}\epsilon^{p-1}\right)  \norm{w_{1} - w_{2}}.  \nonumber    \ee
Therefore, again if $\epsilon$ is sufficiently small, the coefficient of 
$\norm{w_{1} - w_{2}}$ 
in the last formula is less than $1$. 
Hence the Contraction Mapping Theorem applies, 
yielding the existence of a solution $w$ satisfying the condition 
\be    \norm{w}\leq \bar{C}\epsilon.   \label{stimaw}\ee
This concludes the proof of the existence part. 

Now the $C^{1}$-dependence of the function $w$ on $Q$ 
follows from the Implicit Function Theorem; 
see also \cite{AM}, Proposition $8.7$. 
In order to prove the second estimate in $(\ref{derivw})$, let us consider the map 
$H : \R^{3}\times W^{1,2}\left(\Omega_{\epsilon}\right)\times\R\times\R\longrightarrow W^{1,2}\left(\Omega_{\epsilon}\right)\times\R$ 
defined by
\bern H\left(Q, w, \alpha, \epsilon\right)  =
\left( 
\begin{array}{cc} 
      I'_{\epsilon}\left(U_{Q, \epsilon} + w\right) - \alpha\frac{\partial U_{Q, \epsilon}}{\partial Q} \\ 
      \left(w, \frac{\partial U_{Q, \epsilon}}{\partial Q}\right)
\end{array} 
\right).
\nonumber\eern
Then $w\in\left(T_{U_{Q, \epsilon}} Z_{\epsilon}\right)^{\perp}$ is a solution 
of $P_{Q}I'_{\epsilon}\left(U_{Q, \epsilon} + w\right) = 0$ 
if and only if $H\left(Q, w, \alpha, \epsilon\right) = 0$. 
Moreover, for $v\in W^{1,2}\left(\Omega_{\epsilon}\right)$ and $\beta\in\R$, 
there holds
\bern \frac{\partial H}{\partial\left(w, \alpha\right)}\left(Q, w, \alpha, \epsilon\right)\left[v, \beta\right]   &=&
\left( 
\begin{array}{cc} 
      I''_{\epsilon}\left(U_{Q, \epsilon} + w\right)\left[v\right]  - \beta\frac{\partial U_{Q, \epsilon}}{\partial Q} \\ 
      \left(v, \frac{\partial U_{Q, \epsilon}}{\partial Q}\right)
\end{array} 
\right)   \label{stimaH}\\
&=& \left( 
\begin{array}{cc} 
      I''_{\epsilon}\left(U_{Q, \epsilon}\right)\left[v\right]  - \beta\frac{\partial U_{Q, \epsilon}}{\partial Q} \\ 
      \left(v, \frac{\partial U_{Q, \epsilon}}{\partial Q}\right)
\end{array} 
\right)  +    O\left(\norm{w}  + \norm{w}^{p-1}\right).
\nonumber\eern
To prove the last estimate it is sufficient to use the following inequality
\bern 
\vert \left(a + b\right)^{p-1} - a^{p-1} \vert \leq
\left\{ 
\begin{array}{ll} 
         C\left(p\right) \vert b\vert^{p-1}   \quad & \mathrm{for\ } p\leq 2,      \\
         C\left(p\right) \left(\vert b\vert + \vert b\vert^{p-1}\right)   \quad & \mathrm{for\ } p> 2,
 \end{array} 
\right.         
 \nonumber\eern  
for $a, b\in\R$, $\vert a\vert\leq 1$, 
the H\"{o}lder's inequality and the Sobolev embedding. 
Using the invertibility of the operator 
$L_{Q} = P_{Q}\circ I''_{\epsilon}\left(U_{Q,\epsilon}\right)\circ P_{Q}$, 
it is easy to check that $\frac{\partial H}{\partial\left(w, \alpha\right)}\left(Q, 0, 0, \epsilon\right)$ 
is uniformly invertible in $Q$ for $\epsilon$ small. 
Hence, by $(\ref{stimaw})$ and $(\ref{stimaH})$, 
also $\frac{\partial H}{\partial\left(w, \alpha\right)}\left(Q, w, \alpha, \epsilon\right)$ 
is uniformly invertible in $Q$ for $\epsilon$ small. 
As a consequence, by the Implicit Function Theorem, 
the map $Q\mapsto\left(w_{Q}, \alpha_{Q}\right)$ is of class $C^{1}$. 
Now we are in position to provide the norm estimate of $\frac{\partial w\left(\epsilon, Q\right)}{\partial Q}$. 
Differentiating the equation
\be     H\left(Q, w_{Q}, \alpha_{Q}, \epsilon\right) = 0  \nonumber\ee
with respect to $Q$, we obtain
\be   0 = \frac{\partial H}{\partial Q}\left(Q, w, \alpha, \epsilon\right) 
                 + \frac{\partial H}{\partial\left(w, \alpha\right)}\left(Q, w, \alpha, \epsilon\right) \frac{\partial\left(w_{Q}, \alpha_{Q}\right)}{\partial Q}. 
\nonumber\ee
Hence, by the uniform invertibility of $\frac{\partial H}{\partial\left(w, \alpha\right)}\left(Q, w, \alpha, \epsilon\right)$ 
it follows that
\bern  
    \norm{\frac{\partial\left(w_{Q}, \alpha_{Q}\right)}{\partial Q}} &\leq & 
            C \norm{\left( 
\begin{array}{cc} 
      I''_{\epsilon}\left(U_{Q, \epsilon} + w\right)\left[\frac{\partial U_{Q, \epsilon}}{\partial Q}\right]  - \alpha\frac{\partial^{2} U_{Q, \epsilon}}{\partial Q^{2}} \\ 
      \left(w, \frac{\partial^{2} U_{Q, \epsilon}}{\partial Q^{2}}\right) 
\end{array} 
\right)}\nonumber\\
       &\leq & C\left( \norm{I''_{\epsilon}\left(U_{Q, \epsilon} + w\right)\left[\frac{\partial U_{Q, \epsilon}}{\partial Q}\right]}
               +     \vert\alpha\vert\cdot\norm{\frac{\partial^{2}U_{Q, \epsilon}}{\partial Q^{2}}}   
               +     \norm{w} \cdot \norm{\frac{\partial^{2}U_{Q, \epsilon}}{\partial Q^{2}}}\right)   \nonumber\\
        &\leq &    C\left( \norm{I''_{\epsilon}\left(U_{Q, \epsilon} + w\right)\left[\frac{\partial U_{Q, \epsilon}}{\partial Q}\right]}
               +     \vert\alpha\vert + \norm{w} + \epsilon\right).     
\nonumber\eern
Note that $\alpha$, similarly to $w$, satisfies $\vert\alpha\vert\leq C \epsilon$. 
By the estimate in $(\ref{stimaH})$ we obtain
\bern
     \norm{I''_{\epsilon}\left(U_{Q, \epsilon} + w\right)\left[\frac{\partial U_{Q, \epsilon}}{\partial Q}\right]}
         \leq  \norm{I''_{\epsilon}\left(U_{Q, \epsilon}\right)\left[\frac{\partial U_{Q, \epsilon}}{\partial Q}\right]}  
                 +   C\left(\norm{w}  + \norm{w}^{p-1}\right). 
\nonumber\eern
Using the fact that $I''\left(U_{Q}\right)\left[\frac{\partial U_{Q}}{\partial z_{1}}\right]=0$ we obtain
\bern 
       \norm{I''_{\epsilon}\left(U_{Q, \epsilon} + w\right)\left[\frac{\partial U_{Q, \epsilon}}{\partial Q}\right]}   
             \leq   \norm{I''_{\epsilon}\left(U_{Q, \epsilon}\right)\left[\frac{\partial U_{Q}}{\partial z_{1}}\right]  -  
                                   I''\left(U_{Q}\right)\left[\frac{\partial U_{Q}}{\partial z_{1}}\right]}  + C\epsilon 
                 +   C\left(\norm{w}  + \norm{w}^{p-1}\right).     \nonumber\eern
For any $v\in W^{1,2}\left(K\right)$, one finds
\bern      
      \vert   \left(I''_{\epsilon}\left(U_{Q, \epsilon}\right)   -   I''\left(U_{Q}\right)\right) \left[\frac{\partial U_{Q}}{\partial z_{1}}, v\right]  \vert 
                 \leq     p  \int_{K\cap\Omega_{\epsilon}}  \vert  U_{Q, \epsilon} - U_{Q} \vert  \frac{\partial U_{Q}}{\partial z_{1}}  v    
                  +C\epsilon.
\nonumber\eern
The last three formulas implies the estimate for $\frac{\partial w\left(\epsilon, Q\right)}{\partial Q}$. 
This concludes the proof.
\end{proof}

\medskip

Now we can state the following result, which allows us to perform a finite-dimensional 
reduction of problem $(\ref{problem1})$ on the manifold $Z_{\epsilon}$.
\begin{proposition} 
The functional $\Psi_{\epsilon}:Z_{\epsilon}\rightarrow\R$ defined by 
$\Psi_{\epsilon}\left(Q\right) = I_{\epsilon}\left(U_{Q, \epsilon} + w\left(\epsilon, Q\right)\right)$ is of class $C^{1}$ in $Q$ 
and satisfies 
\be        \Psi '_{\epsilon}\left(Q\right) = 0  \qquad \Longrightarrow  \qquad    I'_{\epsilon}\left(U_{Q, \epsilon} + w\left(\epsilon, Q\right)\right) = 0. \nonumber\ee
\label{prop:ridotto}\end{proposition}

\begin{proof} 
This proposition can be proved using the arguments of 
Theorem $2.12$ of \cite{AM}. 
From a geometric point of view, we consider the manifold 
\be \tilde{Z}_{\epsilon} = \left\lbrace U_{Q, \epsilon} + w\left(\epsilon, Q\right) : Q\in\Gamma\right\rbrace.   \nonumber \ee 
Since $(\ref{derivw})$ holds, we have that for $\epsilon$ small 
\be    T_{U_{Q, \epsilon}} Z_{\epsilon}  \sim T_{U_{Q, \epsilon} + w\left(\epsilon, Q\right)}\tilde{Z}_{\epsilon}.  \label{tang}\ee
If $U_{Q, \epsilon} + w\left(\epsilon, Q\right)$ is a critical point of $I_{\epsilon}$ 
constrained on $\tilde{Z}_{\epsilon}$, then 
$I'_{\epsilon}\left(U_{Q, \epsilon} + w\left(\epsilon, Q\right)\right)$ is perpendicular 
to $T_{U_{Q, \epsilon} + w\left(\epsilon, Q\right)}\tilde{Z}_{\epsilon}$, 
and hence, from $(\ref{tang})$, is almost perpendicular to $T_{U_{Q, \epsilon}} Z_{\epsilon}$. 
Since, by construction of $\tilde{Z}_{\epsilon}$, it is 
$I'_{\epsilon}\left(U_{Q, \epsilon} + w\left(\epsilon, Q\right)\right)\in T_{U_{Q, \epsilon}} Z_{\epsilon}$, 
it must be $I'_{\epsilon}\left(U_{Q, \epsilon} + w\left(\epsilon, Q\right)\right)=0$. 
This concludes the proof.
\end{proof}

\subsection{Proof of Theorem $\ref{th:solution}$}
First of all we have 
\bern     \Psi_{\epsilon}\left(Q\right) &=& 
        I_{\epsilon}\left(U_{Q, \epsilon} + w\left(\epsilon, Q\right)\right)  \nonumber\\
            &=&  I_{\epsilon}\left(U_{Q, \epsilon}\right) +  
              I'_{\epsilon}\left(U_{Q, \epsilon}\right)\left[w\left(\epsilon, Q\right)\right]   + O\left(\norm{w\left(\epsilon, Q\right)}^{2}\right). 
\nonumber\eern
Now, using Lemma $\ref{lem:pseudo}$ and the estimate $(\ref{derivw})$ we infer
\be      \Psi_{\epsilon}\left(Q\right) =    I_{\epsilon}\left(U_{Q, \epsilon}\right)  +  O\left(\epsilon^{2}\right).   \nonumber\ee
Hence Lemma $\ref{lem:espansione}$ yields 
\be     \Psi_{\epsilon}\left(Q\right) = C_{0}\alpha\left(Q\right) + O\left(\epsilon\right).   \nonumber\ee
Therefore, if $Q\in\Gamma$ is a local strict maximum or minimum of the function $\alpha$, 
the thesis follows from Proposition $\ref{prop:ridotto}$.

\begin{center}

{\bf Acknowledgements}

\end{center}

\noindent The author has been supported by the project FIRB-Ideas {\em
Analysis and Beyond}, and wants to thank \textit{Andrea Malchiodi} for his great help in the preparation of this paper.


\begin{thebibliography}{99999}
\bibitem[Ad]{Ad} R. A. Adams, \textit{Sobolev Spaces}, Academic Press, New York (1975).
\bibitem[AM]{AM} A. Ambrosetti, A. Malchiodi, \textit{Perturbation Methods and Semilinear Elliptic Problems on $\R^{n}$}, Birkh\"{a}user, Progr. in Math. 240 (2005).
\bibitem[BL]{BL} H. Berestycki, P.L. Lions, \textit{Nonlinear scalar field equations (Part I and Part II)}, Arch. Rat. Mech. Anal. 82 (1983), 313-376.
\bibitem[CH]{CH} R.G. Casten, C.J. Holland, \textit{Instability results for reaction diffusion equations with Neumann boundary conditions}, J. Diff. Eq. 27 (1978), no. 2, 266-273.
\bibitem[Cha]{Cha} I. Chavel, \textit{Eigenvalues in Riemannian geometry}, Academic Press, New York (1984).
\bibitem[DW]{DW} E.N. Dancer, J. Wei, \textit{On the effect of domain topology in a singular perturbation problem}, Topol. Methods Nonlinear Anal. 11 (1998), no. 2, 227-248.
\bibitem[DY]{DY} E.N. Dancer, S. Yan, \textit{Multipeak solutions for a singularly perturbed Neumann problem}, Pacific J. Math. 189 (1999), no. 2, 241-262.
\bibitem[DFW]{DFW} M. Del Pino, P. Felmer, J. Wei, \textit{On the role of the mean curvature in some singularly perturbed Neumann problems}, S.I.A.M. J. Math. Anal. 31 (1999), 63-79.
\bibitem[FW]{FW} A. Floer, A. Weinstein, \textit{Nonspreading wave packets for the cubic Schr\"{o}dinger equation with a bounded potential}, J. Funct. Anal. 69 (1986), 397-408.
\bibitem[GMMP1]{GMMP1} J. Garcia Azorero, A. Malchiodi, L. Montoro, I. Peral \textit{Concentration of solutions for some singularly perturbed mixed problems. Part I: existence results}, Archive Rat. Mech. Anal. 196 (2010), no. 3, 907-950.
\bibitem[GMMP2]{GMMP2} J. Garcia Azorero, A. Malchiodi, L. Montoro, I. Peral \textit{Concentration of solutions for some singularly perturbed mixed problems. Part II: asymptotic of minimal energy solutions}, Ann. Inst. H. Poincar\'{e} Anal. Non Lin\'{e}aire 27 (2010), 37-56.
\bibitem[GM]{GM} A. Gierer, H. Meinhardt, \textit{A theory of biological pattern formation}, Kybernetik (Berlin), 12 (1972), 30-39.
\bibitem[Gri]{Gri} P. Grisvard, \textit{Elliptic problems in nonsmooth domains}, Pitman, London (1985).
\bibitem[Gro]{Gro} H. Groemer, \textit{Geometric applications of Fourier series and spherical harmonics}, Encyclopedia of Mathematics and its Application 61, Cambridge University Press, Cambridge (1996).
\bibitem[Gr]{Gr} M. Grossi, \textit{Some results on a class of nonlinear Schr\"{o}dinger equations}, Math. Z. 235 (2000), no. 4, 687-705.
\bibitem[GPW]{GPW} M. Grossi, A. Pistoia, J. Wei, \textit{Existence of multipeak solutions for a semilinear Neumann problem via non smooth critical point theory}, Calc. Var. Partial Differential Equations 11 (2000), no. 2, 143-175.
\bibitem[Gu]{Gu} C. Gui, \textit{Multipeak solutions for a semilinear Neumann problem}, Duke Math. J. 84 (1996), no. 3, 739-769.
\bibitem[GW]{GW} C. Gui, J. Wei, \textit{Multiple interior peak solutions for some singularly perturbed Neumann problems}, J. Differential Equations 158 (1999), no. 1, 1-27.
\bibitem[GW1]{GW1} C. Gui, J. Wei, \textit{On multiple mixed interior and boundary peak solutions for some singularly perturbed Neumann problems}, Canad. J. Math. 52 (2000), no. 3, 522-538.
\bibitem[GWW]{GWW} C. Gui, J. Wei, M. Winter, \textit{Multiple boundary peak solutions for some singularly perturbed Neumann problems}, Ann. Inst. H. Poincar\'{e} Anal. Non Lin\'{e}aire 17 (2000), no. 1, 47-82.
\bibitem[Ho]{Ho} E.W. Hobson, \textit{The theory of Spherical and Ellipsoidal Harmonics}, Chelsea Pub. Co. (1955).
\bibitem[Kw]{Kw} M.K. Kwong, \textit{Uniqueness of positive solutions of $\Delta u -u + u^{p} = 0$ in $\R^{n}$}, Arch. Rat. Mech. Anal. 105 (1989), 243-266.
\bibitem[Li]{Li} Y.Y. Li, \textit{On a singularly perturbed equation with Neumann boundary conditions}, Comm. Partial Differential Equations 23 (1998), no. 3-4, 487-545.
\bibitem[LN]{LN} Y.Y. Li, L. Nirenberg \textit{The Dirichlet problem for singularly perturbed elliptic equation}, Comm. Pure Appl. Math. 51 (1998), 1445-1490.
\bibitem[LNT]{LNT} C.S. Lin, W.M. Ni, I. Takagi, \textit{Large amplitude stationary solutions to a chemotaxis systems}, J. Differential Equations 72 (1988), 1-27.
\bibitem[Ma]{Ma} A. Malchiodi, \textit{Concentration of solutions for some singularly perturbed Neumann problems}, Geometric analysis and PDEs, 63--115, Lecture Notes in Math., 1977, Springer, Dordrecht (2009).
\bibitem[Mat]{Mat} H. Matano, \textit{Asymptotic behavior and stability of solutions of semilinear diffusion equations}, Publ. Res. Inst. Math. Sci. 15 (1979), 401-454.
\bibitem[Mu]{Mu} C. M\"{u}ller, \textit{Analysis of spherical symmetries in euclidean spaces}, Applied Mathematical Sciences 129, Springer-Verlag, New York (1998).
\bibitem[Mu1]{Mu1} C. M\"{u}ller, \textit{Spherical Harmonics}, Lecture Notes in Math. 17, Springer-Verlag, Berlin, Heidelberg, New York (1966).
\bibitem[Ni]{Ni} W.M. Ni, \textit{Diffusion, cross-diffusion, and their spike-layer steady states}, Notices Amer. Math. Soc. 45 (1998), no. 1, 9-18.
\bibitem[NPT]{NPT} W.M. Ni, X.B. Pan, I. Takagi, \textit{Singular behavior of least-energy solutions of a semilinear Neumann problem involving critical Sobolev exponents}, Duke Math. J. 67 (1992), no. 1, 1-20.
\bibitem[NT1]{NT1} W.M. Ni, I. Takagi, \textit{On the shape of least-energy solution to a semilinear Neumann problem}, Comm. Pure Appl. Math. 41 (1991), 819-851.
\bibitem[NT2]{NT2} W.M. Ni, I. Takagi, \textit{Locating the peaks of least-energy solutions to a semilinear Neumann problem}, Duke Math. J. 70 (1993), 247-281.
\bibitem[NTY]{NTY} W.M. Ni, I. Takagi, E. Yanagida, \textit{Stability of least energy patterns of the shadow system for an activator-inhibitor model. Recent topics in mathematics moving toward science and engineering}, Japan J. Indust. Appl. Math. 18 (2001), no. 2, 259-272.
\bibitem[NW]{NW} W.M. Ni, J. Wei, \textit{On the location and profile of spike-layer solutions to singularly perturbed semilinear Dirichlet problems}, Comm. Pure Appl. Math. 48 (1995), no. 7, 731-768.
\bibitem[Sh]{Sh} J. Shi, \textit{Semilinear Neumann boundary value problems on a rectangle}, Trans. Amer. Math. Soc. 354 (2002), 3117-3154.
\bibitem[St]{St} W.A. Strauss, \textit{Existence of solitary waves in higher dimensions}, Comm. Math. Phys. 55 (1977), 149-162.
\bibitem[Tu]{Tu} A.M. Turing, \textit{The chemical basis of morphogenesis}, Phil. Trans. Royal Soc. London, Series B, Biological Sciences, 237 (1952), 37-72.
\bibitem[Wa]{Wa} Z.Q. Wang, \textit{On the existence of multiple, single-peaked solutions for a semilinear Neumann problem}, Arch. Rational Mech. Anal. 120 (1992), no. 4, 375-399.
\bibitem[We]{We} J. Wei, \textit{On the boundary spike layer solutions of a singularly perturbed semilinear Neumann problem}, J. Differential Equations 134 (1997), no. 1, 104-133.
\bibitem[We1]{We1} J. Wei, \textit{On the construction of single-peaked solutions to a singularly perturbed semilinear Dirichlet problem}, J. Differential Equations 129 (1996), no. 2, 315-333.
\end{thebibliography}
\end{document}